\documentclass[12pt]{amsart}%{article}
\usepackage{amsmath,amstext,amsfonts,
amsthm,amssymb,hyperref,amscd,wasysym,stmaryrd,bbm,xcolor}
\usepackage[utf8]{inputenc}
\usepackage[OT1]{fontenc}
\usepackage{latexsym}
\usepackage[all, 2cell]{xypic}

\usepackage{tikz}
\usetikzlibrary{trees, positioning, calc}
\SelectTips{cm}{}

\usepackage{eucal}
\usepackage[all]{xy}                                         
\swapnumbers
\newtheorem{thm}[subsubsection]{Theorem}
\newtheorem{lem}[subsubsection]{Lemma}
\newtheorem{prp}[subsubsection]{Proposition}

\newtheorem*{Lem}{Lemma}
\newtheorem*{Prp}{Proposition}

{  \theoremstyle{definition}
           \newtheorem{dfn}[subsubsection]{Definition}
           
           \newtheorem{rem}[subsubsection]{Remark}

           \newtheorem*{Rem}{Remark}
           
}
\newcommand{\VH}[1]{{\tt{\color{blue}{#1}}}}
\newcommand{\IM}[1]{{\tt{\color{red}{#1}}}}
\newcommand{\act}{\mathrm{act}}

\newcommand{\Aut}{\operatorname{Aut}}

\newcommand{\Cat}{\mathtt{Cat}}

\newcommand{\colim}{\operatorname{colim}}

\newcommand{\Coop}{\mathtt{Coop}}

\newcommand{\CS}{\mathtt{CS}} % complete Segal

\newcommand{\dSet}{\mathtt{dSet}}
\newcommand{\eq}{\mathrm{eq}}

\newcommand{\Fin}{\mathit{Fin}}
\newcommand{\Fun}{\operatorname{Fun}}

\newcommand{\Hom}{\mathrm{Hom}} % straight Hom
\newcommand{\Ho}{\operatorname{Ho}}

\newcommand{\shom}{\mathcal{H}\!\mathit{om}} % calligraphic Hom for ssets
\newcommand{\id}{\mathrm{id}}
\newcommand{\inn}{\mathit{in}}

\newcommand{\Left}{\mathtt{Left}}

\newcommand{\Map}{\operatorname{Map}}

\newcommand{\odSet}{\mathtt{odSet}}

\newcommand{\op}{\mathrm{op}}
\newcommand{\Op}{\mathtt{Op}}
\newcommand{\DOp}{\mathtt{DOp}}
\newcommand{\LOp}{\mathtt{LOp}}

\newcommand{\out}{\mathit{out}}

\newcommand{\rlarrows}{\substack{\longrightarrow\\ \longleftarrow}}

\newcommand{\SM}{\mathtt{SM}}
\newcommand{\Seg}{\mathtt{Seg}}

\newcommand{\Set}{\mathtt{Set}}

\newcommand{\sSet}{\mathtt{sSet}}

\newcommand{\Triv}{{\mathcal{T}\mathrm{riv}}}
\newcommand{\Tw}{\mathtt{Tw}}

\newcommand{\wt}{\widetilde}

\newcommand{\Alg}{\mathtt{Alg}}

\newcommand{\bF}{\mathbb{F}}

\newcommand{\cC}{\mathcal{C}}
\newcommand{\ccD}{\mathcal{D}}

\newcommand{\cO}{\mathcal{O}}
\newcommand{\cP}{\mathcal{P}}
\newcommand{\cQ}{\mathcal{Q}}

\newcommand{\cS}{\mathcal{S}}

\newcommand{\cX}{\mathcal{X}}

\newcommand{\F}{\mathbb{F}}

\begin{document}

\title[]{On the equivalence of Lurie's $\infty$-operads and dendroidal 
$\infty$-operads}
\author{Vladimir Hinich}
\address{Department of Mathematics, University of Haifa,
Mount Carmel, Haifa 3498838,  Israel}
\email{hinich@math.haifa.ac.il}
\author{Ieke Moerdijk}
\address{Department of Mathematics, University of Utrecht,
Budapestlaan 8, Utrecht,  Netherlands}
\email{i.moerdijk@uu.nl}

\begin{abstract}
In this paper we prove the equivalence of two symmetric monoidal 
$\infty$-categories of $\infty$-operads, the one defined in 
Lurie~\cite{L.HA} and the one based on dendroidal spaces.
\end{abstract}
\maketitle
	
\section{Introduction}
\subsection{}

% ***********************************************

% ***************************************************

In this paper we return to the question of the comparison of various 
notions of $\infty$-operad occurring in the literature. One such notion 
is the one defined by Lurie in terms of simplicial sets over the nerve 
of the category of finite pointed sets, see~\cite{L.HA}, Sect.~2; 
another is the one defined in terms of dendroidal sets, or dendroidal 
spaces~\cite{CM}. Lurie's $\infty$-category and the dendroidal one are 
in fact symmetric monoidal $\infty$-categories where the monoidal structures resemble the Boardman-Vogt tensor product of operads. The $\infty$-categories themselves are underlying 
Quillen model structures. The dendroidal model category has been shown in 
\cite{CM} to be Quillen equivalent to the model category of classical 
simplicial or topological operads.  

A first such comparison was made in~\cite{HHM}, where it was shown that 
if one restricts oneself to operads without constants, the Lurie model 
and the dendroidal one are (Quillen) equivalent at the level of model 
categories. Moreover, this equivalence respects the monoidal structure 
of the associated homotopy categories, which is a shadow of the much 
richer structure of symmetric monoidal $\infty$-category.

In a long paper~\cite{B}, Barwick constructs another $\infty$-category based on his notion of operator category, and   proves this $\infty$-category to be equivalent to Lurie's 
version mentioned above. A next comparison was studied in~\cite{CHH}, where the 
dendroidal model was shown to be equivalent to Barwick's version at the 
level of $\infty$-categories. Combined with Barwick's equivalence, this 
gives a composed equivalence between Lurie's $\infty$-category and the 
dendroidal one, now avoiding the condition on the absence of constants 
of~\cite{HHM}. However, the comparison of~\cite{CHH} does not address 
the question of equivalence as symmetric monoidal $\infty$-categories.

The goal of this paper is to prove a relatively direct and explicit 
equivalence between two symmetric monoidal $\infty$-categories. One is 
the $\infty$-category $\LOp$ (short for ``Lurie operads'') underlying 
Lurie's model category, the other is the $\infty$-category $\DOp$ 
(for ``dendroidal operads'') underlying the dendroidal model category. 
To give the reader a rough idea already at this stage, we remark that 
our proof is based on a functor from level forests to forests, 
denoted

$$\omega: \bF \to\Phi,$$ 
see~\ref{sss:Phi} and \ref{ss:main} for the notation.
$\LOp$ is an $\infty$-category of presheaves on $\bF$ and $\DOp$ is one 
on $\Phi$, and the equivalence is simply realized by the functors
\footnote{As a mnemonic aid, $\delta$ stands for ``dendrification'',
$\lambda$ for ``Luriefication''.
}
\begin{align*}
\lambda:\DOp\to\LOp,&\  \ \lambda(D)(A) = \Map_\DOp(\omega(A),D),\\
\delta:\LOp\to\DOp, & \ \ \delta(L)(F)  = \Map_\LOp(i(F), L).
\end{align*}

Here $A, D, F$ and $L$ are objects of $\bF$, $\DOp$, $\Phi$ and 
$\LOp$ respectively, and $i$ denotes the embedding of objects of 
$\Phi$ as free operads in $\LOp$; see \ref{ss:main} below for detailed definitions. Our main theorem can then be stated as follows:

\begin{thm}
The functors $\lambda$ and $\delta$ define an equivalence of symmetric monoidal $\infty$-categories
$$
\lambda:\DOp\rlarrows\LOp:\delta.
$$
\end{thm}

The other comparison proofs mentioned above are also based on the same 
functor from level forests to forests, but there are several important 
differences. First of all, our result is an equivalence of symmetric 
monoidal $\infty$-categories, not just of $\infty$-categories. To prove 
this sharper result, we use a colax symmetric monoidal structure on the 
category of copresheaves on an operad~\footnote{Or, as it appears in 
our paper, the category of presheaves on an anti-operad.}, which may be of 
independent interest. Secondly, our proof uses the category of algebras 
of an operad in two essential ways. We use the comparison theorem 
of Pavlov and Scholbach~\cite{PS} which states that the $\infty$-category underlying the 
category of simplicial algebras over a $\Sigma$-free  
operad $P$ in  sets is equivalent to the $\infty$-category of algebras over the 
associated 
$\infty$-operad $\ell(P)$ in the $\infty$-category of spaces, and 
similarly for algebras in a symmetric monoidal model category $\cC$ and 
its underlying $\infty$-category $\cC_\infty$. We state this result 
somewhat cryptically as
$$
\Alg_P(\cC)_\infty = \Alg_{\ell(P)}(\cC_\infty),
$$
see~\ref{sss:PSrectification} for a precise 
formulation. (This result is analogous to an earlier result 
for linear operads proved in~\cite{H.R}.) Secondly, we prove and use the 
following reconstruction theorem for $\infty$-operads, stating that a 
map $\cP\to\cQ$ between $\infty$-operads which is essentially surjective 
on colors is an equivalence whenever it induces an equivalence between 
the associated $\infty$-categories of algebras in the $\infty$-category 
of spaces; see Theorem~\ref{thm:reconstruction} below. (By the result of \cite{PS} just mentioned this  implies the analogous known result 
for $\Sigma$-cofibrant simplicial operads, 
see~\cite{CG}.)

\subsection{}

To conclude this introduction, let us briefly sketch the contents of 
this paper. In Section~\ref{sec:preliminary}, we fix some conventions 
about our use of $\infty$-categorical language, and introduce the 
$\infty$-categories $\DOp$ and $\LOp$ featuring in our main theorem above. In Section~\ref{sec:equivalenceLOPDOP}, we 
state and prove a weaker form of the main theorem, ignoring the 
symmetric monoidal structure for the moment. The proof uses a lemma 
which is based on the reconstruction theorem, which we postpone until 
Section~\ref{sec:opalgebras} where we discuss algebras over an 
$\infty$-operad. In the
final Section~\ref{sec:monoidal} we address the different symmetric monoidal structures 
involved, and prove that they are respected by the functors $\delta$
and $\lambda$. The structure on $\DOp$ is defined in terms of shuffles of trees about which we explain some basic facts in the Appendix.

\subsection{Acknowledgements}

A large part of the paper was written during V.H.'s (pre-pandemic) visits to University of Utrecht. We thank Rune Haugseng for  helpful correspondence and Matija Basi\'c for help with the pictures.  We are very grateful to the referees whose remarks
helped to improve the exposition.
The work of V.H. was supported by ISF 786/19 grant.

% end of Ieke's version

\section{Preliminary definitions}
\label{sec:preliminary}

\subsection{$\infty$-categorical conventions}
\label{ss:conventions}

 We present here some basic notation and discuss a few standard recipes 
for working with $\infty$-categories. 

In what follows the word ``category'' means $\infty$-category, operad means $\infty$-operad, functor means $\infty$-functor and so on.
If we wish to emphasize that an $\infty$-category (or an 
$\infty$-operad) can be modelled by a {\sl strict} inner Kan complex,
we sometimes refer to it as a {\sl conventional} category (or operad).

The most basic category is the category of spaces $\cS$
underlying the Quillen model category of simplicial sets. Given two categories $\cC$, $\ccD$, there is a category of functors $\Fun(\cC,\ccD)$
satisfying the standard equivalence
$$
\Map(\cX,\Fun(\cC,\ccD))=\Map(\cX\times\cC,\ccD).
$$

The category of categories $\Cat$ can be realized as a full subcategory
of the category of simplicial spaces $P(\Delta)=\Fun(\Delta^\op,\cS)$,
spanned by the simplicial presheaves $X:\Delta^\op\to\cS$ 
satisfying
the $\infty$-categorical variant of the Segal and completeness 
properties, see~\cite{R}:
\begin{itemize}
\item[1.] For any $n$ the natural map
$$
X_n\to X_1\times_{X_0}\cdots_{X_0}\times X_1
$$
is an equivalence.
\item[2.]  The map $\Map(J,X)\to\Map(*,X)$
induced by a map $*\to J$, is an equivalence. ($J\in P(\Delta)$ is the presheaf corresponding to the category having two objects and a unique isomorphism between them.) 
\end{itemize}

%~\footnote{This is a $\infty$-categorical reformulation of the Rezk's complete Segal model for $\infty$-categories,
%.}
This approach allows one to define the opposite of a category
defined by a simplicial space $\cC:\Delta^\op\to\cS$ as the composition
of $\cC$ with the functor $\op:\Delta\to\Delta$ reversing the ordering of a finite totally ordered set.

\subsubsection{}

Another endofunctor of $\Delta$, carrying $[n]$ to the join
$[n]^\op\star[n]=[2n+1]$, gives the construction of $\Tw(\cC)$, the category of twisted arrows in $\cC$. The canonical projection
$$\Tw(\cC)\to\cC^\op\times\cC$$
is a left fibration; that is, it is classified by a functor
$$
\wt Y:\cC^\op\times\cC\to\cS.
$$
This functor can be rewritten as the Yoneda embedding $Y:\cC\to P(\cC)=
\Fun(\cC^\op,\cS)$.

Dealing with $\infty$-categories requires extra care when writing formulas. It is in general not allowed to define functors by
describing them on objects and arrows as there is no way to describe all
required compatibilities. However, some standard formulas
do define functors. For instance, given a functor $f:\cC\to\ccD$,
one has a functor $\wt f:\ccD^\op\times\cC\to\cS$ defined by the 
formula $\wt f(d,c)=\Map_\ccD(d,f(c))$. This formula just means that
$\wt f$ is defined as the composition 
$$
\ccD^\op\times\cC\to\ccD^\op\times\ccD\stackrel{\wt Y}{\to}\cS.
$$
\subsubsection{}

A map of spaces $f:X\to Y$ (for example, modelled by Kan simplicial sets) exhibits $X$ as {\sl a  subspace of $Y$} if $f$ induces an equivalence of $X$ with a union of a subset of connected components of $Y$. This
notion generalizes to any category $\cC$: an arrow $f:c\to d$ is mono,
if  for any $x\in \cC$ the induced map of spaces $\Map_\cC(x,c)\to\Map_\cC(x,d)$
is an inclusion of a subspace. In the case when $\cC=\Cat$, we get the notion of
subcategory: it is defined by a subspace of objects, and a subspace
of morphisms for each pair of objects.

\subsubsection{Subfunctor} 
\label{sss:subfunctor}
The notion of subspace mentioned above allows one to construct
a subfunctor of a given functor.
In this context the following elementary result is useful
(see~\cite{H.Lec}, 9.2.3). 

\begin{Prp} 
Let $F:\cC\to\ccD$ be a functor. Let, for each $x\in \cC$, a subobject $G_x$ of $F(x)$ be given, so that for each $a:x\to y$ the composition $G_x\to F(x)\to F(y)$ factors through $G_y$. Then the collection of subobjects $G_y$ uniquely glues into a subfunctor
$G:\cC\to\ccD$.
\end{Prp}

\subsubsection{} 
\label{sss:loc}
A marked category is a pair 
$\cC^\natural=(\cC,\cC^\circ)$ where $\cC$ is a category
and $\cC^\circ$ is a subcategory of $\cC$ containing $\cC^\eq$, the maximal subspace of $\cC$. 
The category $\Cat^+$ of marked categories is defined
as the full subcategory of $\Fun([1],\Cat)$ spanned by the embeddings
$\cC^\circ\to\cC$. 
We denote by $\cC^\flat=(\cC,\cC^\eq)$ the category $\cC$ endowed with the minimal marking. The embedding
$
\Cat\to\Cat^+
$
carrying $\cC$ to $\cC^\flat$ has a left adjoint called localization 
and denoted
\begin{equation}
\label{eq:loc}
L:\Cat^+\to\Cat.
\end{equation}
For marked categories represented by a pair of simplicial categories,
the localization is represented by the Dwyer-Kan construction.

\subsubsection{}
\label{sss:bousfield}
Many important $\infty$-categories appear as the ones {\sl underlying} 
model categories. One such is the standard model structure on 
simplicial sets modelling $\cS$. Another one is the complete Segal model 
for the $\infty$-category of $\infty$-categories $\Cat$ that has already been mentioned above. 

The $\infty$-category of a model
category is obtained by a general localization construction as described in ~\ref{sss:loc} which does not enjoy very nice properties. 
Fortunately, one can often present the $\infty$-category underlying
a model category as a Bousfield localization~\footnote{A localization
that has a fully faithful right adjoint} of a certain $\infty$-category 
of presheaves of spaces. For instance, 
$\Cat$ can be presented as a Bousfield localization of the 
$\infty$-category $P(\Delta)$ of simplicial spaces. As explained above,
this is an $\infty$-categorical reformulation of the fact \cite{R}
that the model category of  complete Segal spaces is obtained by
a Bousfield localization  (in the sense of model categories) from the 
Reedy model structure on bisimplicial sets. Below we will present the 
$\infty$-categories $\DOp$ and $\LOp$ in a similar way, see \ref{sss:dop} and \ref{sss:lop-bousfield}.

Note that by a result of Dugger~\cite{D}, any $\infty$-category 
underlying a combinatorial model category is in fact equivalent to 
such a localization of a category $P(\cC)$ of simplicial presheaves.

Let $\CS\subset\Seg\subset P(\Delta)$ denote the full subcategories
of $P(\Delta)$ spanned by the complete Segal spaces and by all Segal spaces, respectively. The following easy observation  will be used below.

\begin{lem}
\label{lem:completeness-fiber}
Let $f:X\to B$ be an arrow in $\Seg$ with $B\in\CS$. Then $X\in\CS$
iff the fibers of $f$ are in $\CS$.
\end{lem}
\begin{proof}
Recall from \ref{ss:conventions} that a Segal space $X$ is complete iff the map $\Map(J,X)\to\Map(*,X)$
induced by a map $*\to J$, is an equivalence.
We have a commutative diagram of spaces
$$
\xymatrix{
&\Map(J,X)\ar[r]\ar[d]&\Map(*,X)\ar[d]\\
&\Map(J,B)\ar^\sim[r]&\Map(*,B)
}
$$
and we need to show that the top horizontal arrow is an
equivalence. It is obviously so iff the map of fibers at any $b\in\Map(*,B)$ is an equivalence. These fibers identify with
$\Map(J,X_b)$ and $\Map(*,X_b)$ where $X_b$ denotes the fiber of $f$ 
at $b\in B$.
\end{proof}

\subsection{Dendroidal $\infty$-operads}
\label{ss:dop}

\subsubsection{}
\label{sss:dop}
We begin by recalling the definition of the category $\Omega$ introduced in \cite{MW} and discussed in detail in \cite{HM}, Section 3.2,
see also the Appendix below.
The objects of $\Omega$ are finite trees, allowed to have ``external''
edges attached to just one vertex. One of these external edges is
specified as the {\sl root} of the tree, the other external edges are called leaves. The edges of a tree connected to two vertices are called {\sl internal} or {\sl inner} edges. The category $\Omega$ includes the object $\eta$ consisting of just one edge that is at the same time the root and the leaf. The choice of a root defines an orientation of each edge in the tree, towards the root. This specifies for each vertex $v$ an outgoing edge $\out(v)$ and
a set  $\inn(v)$ of incoming edges. The cardinality of $\inn(v)$
is called the {\sl valence} of the vertex $v$, and it is allowed to be zero. Vertices of valence zero are called {\sl stumps}.
To define the morphisms of the category $\Omega$, we observe that each such tree $T$ defines a (symmetric) colored operad
$o(T)$. The colors of $o(T)$ are the edges of $T$ and its operations are generated by the vertices, each vertex $v$ defining an operation from the set $\inn(v)$ to $\out(v)$. The morphisms $S\to T$ in $\Omega$ are now defined to be the operad maps $o(S)\to o(T)$. In particular, this makes $o$ into a full embedding of $\Omega$ into the category $\Op(\Set)$ of operads in sets.

Alternatively, one may define the morphisms $S\to T$ in $\Omega$ as generated by ``elementary morphisms'': these are 
\begin{itemize}
\item[(a)] isomorphisms $S\stackrel{\sim}{\to}T$;
\item[(b)] degeneracies $S\twoheadrightarrow T$ where $S$ is obtained from $T$ by putting a new vertex in the middle of an edge of $T$;
\item[(c)] inner face maps $S\rightarrowtail T$ where $S$ is obtained from $T$ by contracting an inner edge of $T$; and
\item[(d)] outer face maps where $S$ is obtained by chopping off
an external vertex from $T$ (i.e., a vertex attached to just one inner edge). In addition, if $T$ is a corolla, i.e. a tree
with just one vertex, each edge of $T$ defines an elementary morphism $\eta\to T$.
\end{itemize}

We observe now that there is a full embedding $\iota:\Delta\hookrightarrow\Omega$ of the simplex category $\Delta$ into $\Omega$,
which assigns to each object $[n]$ of $\Delta$ the linear tree
$\iota[n]$ with $n$ vertices (all of valence $1$) and $n+1$ edges. Under this identification of $\Delta$ with a subcategory of $\Omega$, degeneracies, inner and outer faces have their usual meaning.

\subsubsection{}
\label{sss:Phi}
It is convenient to extend the category $\Omega$ to include disjoint unions of trees. To this end we define the
category of forests $\Phi$ similar to $\Omega$, as the full subcategory
of $\Op(\Set)$, spanned by $o(F)$ where $F$ is a disjoint union of trees and $o(F)$ is defined as the coproduct of the operads $o(T_i)$ where $T_i$ are the components of $F$. We will denote by $o:\Phi\to\Op(\Set)$ the full embedding described above, see diagram 
(\ref{eq:thediagram}). The category $\Phi$ can be alternatively defined as the category
obtained from $\Omega$ by formally adjoining  finite coproducts.
(Note that our definition of $\Phi$ is different from that of
\cite{HHM}. First, our category $\Phi$ has an empty forest; and, secondly,
the version of \cite{HHM} has fewer arrows (independence property) than
our version.)

\subsubsection{}
\label{sss:dop}
The ($\infty$-) category of dendroidal operads $\DOp$ is defined as the full subcategory 
of the category $P(\Omega)$ of presheaves (of spaces) spanned by the presheaves
satisfying Segal and completeness properties:
\begin{itemize}
\item[(D1)] For an inner edge $b$ in $T$ let $T_b$ and $T^b$ be the upper and the lower part
of $T$ obtained by cutting $T$ at $b$. Then $X(T)\to X(T_b)\times_{X(b)}X(T^b)$
is an equivalence.
\item[(D2)]  Completeness: $\iota^*(X)\in P(\Delta)$ is a complete
Segal space.
\end{itemize}

It will be convenient for us to realize $\DOp$ as a full subcategory in $P(\Phi)$ spanned by the presheaves
satisfying the above two properties, as well as the  extra (also Segal-type) property.
\begin{itemize}
\item[(D3)] The natural map $X(F)\to\prod X(T_i)$ for a forest $F$ consisting of the trees $T_i$, is an equivalence.
In particular, $X(\emptyset)$ is contractible.
\end{itemize}
It is a standard fact that the category $\DOp$ is equivalent to the one underlying
the model category of dendroidal complete Segal spaces,  or the equivalent one of dendroidal sets,
see~\cite{CM} or \cite{HM}.

\subsection{Lurie $\infty$-operads}
\label{ss:lop}

In what follows $\Fin_*$ denotes the conventional category of finite pointed sets. We denote by $I_*$ the finite pointed set $I\sqcup\{*\}$
and put $\langle n\rangle=\{1,\ldots,n\}_*$.
An arrow $f:I_*\to J_*$ is called {\sl inert} if for any $j\in J$
the set $f^{-1}(j)$ consists of one element.

A functor $p:\cP\to\Fin_*$ is called {\sl fibrous} if the following
conditions are satisfied. In what follows we denote by $\cP_n$ (or
$\cP_I$) the fiber of $p$ at $\langle n\rangle$ (or at $I_*$).
\begin{itemize}
\item[(Fib1)] Any inert arrow $f:I_*\to J_*$ has a cocartesian lifting.
As a result, a functor $f_!:\cP_I\to \cP_J$ is defined
(uniquely up to equivalence).
\item[(Fib2)] For the collection of standard inerts 
$\rho^i:\langle n\rangle\to\langle 1\rangle$ defined by
$(\rho^i)^{-1}(1)=\{i\}$, the maps
$$
\cP_n\stackrel{\rho^i_!}{\to}\cP_1
$$
form a product diagram.
\item[(Fib3)] Let $f:\langle m\rangle\to\langle n\rangle$ be an arrow in 
$\Fin_*$, $x\in\cP_n$ and $r^i:x\to x_i$ the cocartesian liftings of
$\rho^i:\langle n\rangle\to \langle 1\rangle$. Then the natural
map
$$
\Map^f(y,x)\to\prod_i\Map^{\rho^i\circ f}(y,x_i)
$$
is an equivalence for any $y\in\cP_m$.
\end{itemize}

\subsubsection{}
\label{sss:lop}
Cocartesian liftings of inerts in $\Fin_*$ are called inerts in $\cP$.
The category of Lurie operads $\LOp$ is the subcategory of 
$\Cat_{/\Fin_*}$ consisting of fibrous objects and morphisms preserving inerts, see~\cite{L.HA}, 2.3.3.28 or \cite{H.EY}, 2.6.3.
For example, any conventional (colored) operad $P$ in $\Set$ with 
the set of colors $[P]$ defines an object $\ell(P)$ of $\LOp$, see~\cite{L.HA}, 
2.1.1.7 and 2.1.1.22. Its morphisms over a map $\alpha:I_*\to J_*$ in 
$\Fin_*$ are triples $(c,d,p)$ where $c:I\to[P]$, $d:J\to[P]$, and 
$p=(p_j:j\in J)$ where $p_j\in P(c|_{\alpha^{-1}(j)},d(j))$. 
These formulas define a functor $\ell:\Op(\Set)\to\LOp$ identifying $\Op(\Set)$
with the full subcategory of $\LOp$ spanned by fibrous maps 
$p:C\to\Fin_*$ with $C$ a conventional category, see diagram~(\ref{eq:thediagram}).

\subsubsection{}
\label{sss:smlop}

The category $\LOp$ has a symmetric monoidal structure that is induced
from the smash product operation on $\Fin_*$.
Following Lurie~\cite{L.HA}, 2.2.5.9, we will say that a functor
$F:(\Fin_*)^n\to\Fin_*$ is a smash product functor if 
$F(\langle 1\rangle,\dots,\langle 1\rangle)\approx\langle 1\rangle$
and $F$ preserves coproducts in each argument. A smash product functor 
is unique up to a unique isomorphism. An {\sl operad multifunctor}
$(P_1,\ldots,P_n)\to Q$ is defined as a commutative diagram
$$
\xymatrix{
&P_1\times\dots\times P_n\ar[d]\ar^{\quad\quad f}[r] & Q\ar[d] \\
&\Fin_*\times\dots\times\Fin_*\ar^{\quad\quad F}[r]&\Fin_*
}
$$
where $f$ carries $n$-tuples of inerts in $P_1\times\dots\times P_n$ 
to inerts in $Q$ and $F$ is a smash product.
The notion of operad multifunctor defines on $\LOp$ the structure of an 
operad that turns out to be a symmetric monoidal category, 
where the multiple tensor product is defined as the
target of the universal operad multifunctor out of $P_1,\dots,P_n$,
see~\cite{L.HA}, 2.2.5.13.

\subsubsection{}

As explained before, the category $\Cat$ identifies with the full subcategory of $P(\Delta)$
spanned by the complete Segal objects. 
In particular, the (conventional) category $\Fin_*$ can be viewed as an object of $P(\Delta)$, which we still denote by $\Fin_*$.
This identifies
$\Cat_{/\Fin_*}$ with a full subcategory of $P(\Delta)_{/\Fin_*}$.

The following general fact allows one to identify the latter with
$P(\Delta_{/\Fin_*})$.

Let $Y:\cC\to P(\cC)$
be the Yoneda embedding and let $F\in P(\cC)$. We define
\begin{equation}
\label{eq:overcat}
\cC_{/F}=\cC\times_{P(\cC)}P(\cC)_{/F}
\end{equation}
and denote by $p:\cC_{/F}\to P(\cC)_{/F}$ the natural projection.
Then $p$ is fully faithful and we denote by
\begin{equation}
\label{eq:pshriek}
p_!:P(\cC_{/F})\to P(\cC)_{/F}
\end{equation}
the extension of $p$ to a colimit-preserving functor.

\begin{lem}
\label{lem:presheaves-eq}
$p_!$ is an equivalence.
\end{lem}
\begin{proof}
We apply Corollary \cite{L.T}, 5.1.6.11,
to $p_!$. Its restriction $p$ is fully faithful, and the image of any
$\phi: Y(x)\to F$ is absolutely compact as colimits in $P(\cC)_{/F}$ are 
detected by colimits in $P(\cC)$. Finally, it is clear that 
the objects $Y(x)\to F$ generate $P(\cC)_{/F}$ under colimits.
\end{proof}

\begin{Rem}
Note that, by definition (\ref{eq:overcat}), 
$\Delta_{/\Fin_*}=\Delta\times_{P(\Delta)}P(\Delta)_{/\Fin_*}$. Since
$\Fin_*$ and all $[n]$ are conventional categories and since the images 
of conventional categories form a full subcategory of $P(\Delta)$, 
$\Delta_{/\Fin_*}$ is equivalent to the conventional category of 
simplices in $\Fin_*$. 

\end{Rem}
In this paper we denote
\begin{equation}
\label{eq:bF}
\bF=\Delta_{/\Fin_*}.
\end{equation}
 Thus $\LOp$ can be identified with a (non-full) subcategory of $P(\F)$.
 
\subsubsection{}
\label{sss:lop-bousfield}
Here is another presentation of $\LOp$, this time as a Bousfield localization. 
The category $\Fin_*$ has a marking
defined by the collection of inert arrows. This marked category is 
denoted by $\Fin_*^\natural$. There is a fully faithful functor 
$\LOp\to\Cat^+_{/\Fin_*^\natural}$ carrying a fibrous
$p:\cO\to\Fin_*$ to the marked category $\cO^\natural$ over 
$\Fin_*^\natural$, with the marking on $\cO$ defined by the inerts. 
By \cite{H.EY}, 2.6.4 (based on \cite{L.HA}, B.0.20),
$\LOp$ is the Bousfield localization with respect to the class of
{\sl operadic equivalences}. An operadic equivalence
is defined as a map $f:X\to Y$ in 
$\Cat^+_{/\Fin_*^\natural}$ inducing an equivalence
$\Map(Y,\cO)\to\Map(X,\cO)$ for any fibrous $\cO$.

\subsubsection{Cocartesian arrows}
\label{sss:cocartesian}
Let $F\in\Cat$. As before, we will identify $F$ with the corresponding
complete Segal space in $P(\Delta)$.
Lemma~\ref{lem:presheaves-eq} defines a full embedding of 
$\Cat_{/F}$ into $P(\Delta_{/F})$. Let $p:X\to F$ be a category over $F$
and let $\cX\in P(\Delta_{/F})$ be the corresponding presheaf. By definition, for $A:[n]\to F$, 
$$\cX(A)=\Map([n],X)\times_{\Map([n],F)}\{A\}.$$
Fix $\alpha:[1]\to F$ and let $a\in\cX(\alpha)$.
We denote by the same letters $a:x\to y$ and $\alpha:\bar x\to \bar y$
the arrows in $X$ and in $F$.
The following lemma is a direct reformulation of the cocartesian property 
of the arrow $a$ in our language.

\begin{Lem}
The arrow $a\in\cX(\alpha)$ is $p$-cocartesian iff for any 
$\sigma:[2]\to F$ with $d_2\sigma=\alpha$, $d_0\sigma=\beta$,
$d_1\sigma=\gamma$, the map 
$$
\{y\}\times_{\cX(\bar y)}\cX(\beta)\to
\{x\}\times_{\cX(\bar x)}\cX(\gamma)
$$
defined as a composition
$$
\{y\}\times_{\cX(\bar y)}\cX(\beta)\to
\{x\}\times_{\cX(\bar x)}\cX(\alpha)\times_{\cX(\bar y)}\cX(\beta)
\stackrel{\sim}{\leftarrow}\{x\}\times_{\cX(\bar x)}\cX(\sigma)
\stackrel{d_1}{\to}
\{x\}\times_{\cX(\bar x)}\cX(\gamma), 
$$
is an equivalence.
\end{Lem}
\qed

\section{Equivalence of $\LOp$ with $\DOp$} 
\label{sec:equivalenceLOPDOP}

\subsection{}
\label{ss:main}

In this section we will construct an equivalence of ($\infty$-) categories between $\LOp$ and $\DOp$. (It will be upgraded
to an equivalence of symmetric monoidal categories
after some more work.) The construction is based on a functor
$$
\omega:\Delta_{/\Fin_*}=\F\to\Phi,
$$
see diagram (\ref{eq:thediagram}),
which we will define first. The definition of $\omega$
is a variant of the one in \cite{HHM} which dealt with open trees and forests only.

\subsubsection{} 
Consider an object $A:[n]\to\Fin_*$ of $\F$, i.e. a sequence 
$$
A_{0*}\stackrel{\alpha_1}{\to}A_{1*}\to\ldots
\stackrel{\alpha_n}{\to}A_{n*}
$$
of maps between pointed sets. We write $\alpha_{ij}:A_{j*}\to A_{i*}$ for the composition 
$\alpha_i\circ\ldots\circ\alpha_{j+1}$ (for $i\geq j$).
The set of edges of the forest $\omega(A)$ is the disjoint union
$\coprod A_i$ of the sets $A_i$. This set carries a partial order defined for $a\in A_i$ and $b\in A_j$ by
$$
a\leq b\textrm{ iff } j\leq i\textrm{ and }\alpha_{ij}(b)=a.
$$
The roots of $\omega(A)$ are the edges minimal in the above order.
For each $a\in A_i$ in this set of edges with $i>0$, there is a unique
vertex $v_a$ in the forest $\omega(A)$ immediately above $a$. The edge
$a$ is the outgoing edge of $v_a$, while $\inn(v_a)=\alpha^{-1}_i(a)$.
In particular, the set of leaves in the forest can be identified with $A_0$. The set of roots of $\omega(A)$ consists of the elements of $A_n$
together with the elements of $A_i$ sent to the basepoint $*$ under $\alpha_{i+1}:A_i\to A_{i+1}$ for $i=0,\ldots,n-1$.

Here is an example
of the forest corresponding to the map 
$\langle 4\rangle\stackrel{\alpha_1}{\to}\langle 3\rangle\stackrel{\alpha_2}{\to}\langle 1\rangle$
with $\alpha_1(1)=1=\alpha_1(2)$, $\alpha_1(3)=3=\alpha_1(4)$,
$\alpha_2(1)=1=\alpha_2(2)$, $\alpha_2(3)=*$.

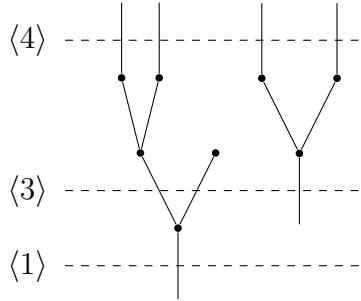
\[
\begin{tikzpicture} 
[level distance=10mm, 
every node/.style={fill, circle, minimum size=.1cm, inner sep=0pt}, 
level 1/.style={sibling distance=20mm}, 
level 2/.style={sibling distance=10mm}, 
level 3/.style={sibling distance=5mm}]

%left tree
\node (lefttree)[style={color=white}] {} [grow'=up] 
child {node (level1) {} 
		child{node (level2) {} 
			child
			child
	}
	child{ node {}}	
};

%right tree
\node (righttree)[style={color=white}, right = 1.5cm of lefttree] {};
\node (righttreestart)[style={color=white}, above = 0.88cm of righttree] {} [grow'=up] 
child {node {} 
	child
	child
};

\tikzstyle{every node}=[]

%lines
\draw[dashed] ($(level1) + (-1.5cm, -0.5)$) -- ($(level1) + (2.5cm, -0.5cm)$);
\draw[dashed] ($(level1) + (-1.5cm, 0.5cm)$) -- ($(level1) + (2.5cm, 0.5cm)$);
\draw[dashed] ($(level1) + (-1.5cm, 1.5cm)$) -- ($(level1) + (2.5cm, 1.5cm)$);

%labels
\node at ($(level1) + (-2cm, 1.5cm)$) {$\langle 4 \rangle$};
\node at ($(level1) + (-2cm, 0.5cm)$) {$\langle 3 \rangle$};
\node at ($(level1) + (-2cm, -.5cm)$) {$\langle 1 \rangle$};

\end{tikzpicture} 
\]

This defines $\omega:\F\to\Phi$ on objects. It extends to morphisms in 
the obvious way: a face map $d_iA\to A$ induces a morphism $\omega(d_iA)
\to\omega(A)$ which on each component tree is a composition of faces; and a degeneracy map $A\to s_iA$ induces a morphism $\omega(A)\to\omega(s_iA)$ which is a composition of degeneracies.

Note the following property of $\omega$.
\begin{lem}
\label{lem:retract}
Any forest $F\in\Phi$ is a retract of some $\omega(A)$ for some $A\in\F$.
\end{lem}
\begin{proof}
In order to present a forest $F$ as $\omega(A)$, one has to assign
a nonnegative number $h(a)$ to each edge $a$ so that $h(a)=h(b)-1$ for $a$ immediately under $b$ and so that the $a$ is a leaf precisely when
$h(a)=0$. The first condition is achieved easily; to achieve the second,
one may need to enlarge the forest $F$ slightly and construct a forest $F'$ by adjoining a sequence of unary edges on top of leaves of $F$. Then $F'$ is of the form $\omega(A)$ and $F$ is a retract of $F'$.

\end{proof}

\subsubsection{}
The functor $\omega$ defines an adjoint pair
$$
\omega_!:P(\bF)\rlarrows P(\Phi):\omega^*.
$$
The functor $\lambda:\DOp\to P(\bF)$ is defined as the restriction of 
$\omega^*$ to $\DOp$. This means that for $D\in\DOp$ and $A\in\bF$,
\begin{equation}
%\label{eq:lambd}
\lambda(D)(A)=\Map_{P(\Phi)}(\omega(A),D).
\end{equation}
Define $i:\Phi\to\LOp$
%, see diagram (\ref{eq:thediagram}), 
as the composition of $o:\Phi\to\Op(\Set)$
with the embedding $\ell:\Op(\Set)\to\LOp$ discussed in~\ref{sss:lop}.
The functor $i$ determines a functor $\LOp\times\Phi^\op\to\cS$
that yields $\delta:\LOp\to P(\Phi)$  by adjunction.
This means that, for $P\in\LOp$ and $F\in\Phi$,
one has
\begin{equation}
\label{eq:delta}
\delta(P)(F)=\Map_\LOp(i(F),P).
\end{equation}

\begin{thm}
\label{thm:equivalence}
The functors defined above give a pair of quasi-inverse functors
\begin{equation}
\label{eq:deltalambda}
\delta:\LOp\rlarrows\DOp:\lambda.
\end{equation}
\end{thm}

\subsubsection{}
In Section~\ref{sec:monoidal} we will extend this equivalence to an equivalence of symmetric monoidal categories.
The proof of Theorem~\ref{thm:equivalence} is presented in 
\ref{ss:inDOp}--\ref{ss:phipsi} below.
We will first of all verify that $\delta(P)\in\DOp$ for any $P\in\LOp$
and that $\lambda$ carries $\DOp$ to $\LOp$.
Then we will construct equivalences $\lambda\circ\delta\to\id$ and
$\id\to\delta\circ\lambda$.

\subsection{The functor $\delta$ has image in $\DOp$}
\label{ss:inDOp}
The category $\DOp$ is a full subcategory of $P(\Phi)$, so we only
have to verify that, for $L\in\LOp$, the presheaf $\delta(L)$ satisfies the conditions (D1), (D2) and (D3). The functor $i$ carries 
a finite coproduct of forests  to the corresponding coproduct in $\LOp$
(since both $o$ and $\ell$ above preserve coproducts).
Also, for an inner edge $b$ of a tree $T$, Proposition~\ref{prp:cuttingtree} below claims that 
$i(T)$ is the colimit of the diagram $i(T^b)\leftarrow i(b)\to i(T_b)$,
where $T^b$ and $T_b$ are two halves of the tree $T$ obtained by cutting $T$ along $b$.
These two facts immediately prove the conditions (D1) and (D3). 
It remains to verify (D2). The simplicial space $\iota^*\circ\delta(L)$ is just the image of $L$ under the functor
$P(\Delta_{/\Fin_*})\to P(\Delta)$ defined by $\langle 1\rangle\in\Fin_*$.
It is complete as $L$ represents a category over $\Fin_*$.

\subsection{The functor $\lambda$ has image in $\LOp$.}
\subsubsection{}
Let us, first of all, verify that $\lambda(D)\in\Cat_{/\Fin_*}$ for any $D\in\DOp$.
We have to verify that $\lambda(D)$, considered as an object of
$P(\Delta_{/\Fin_*})$,  
satisfies the Segal condition and is complete.

For $A:[n]\to\Fin_*$ we denote by $A_i$ the composition
$\{i\}\to[n]\stackrel{A}{\to}\Fin_*$ and by $A_{i-1,i}$ the composition
$[1]\stackrel{\{i-1,i\}}{\longrightarrow}[n]\stackrel{A}{\to}\Fin_*$.
The Segal condition for $\lambda(D)$ means that the natural map 
$$
\lambda(D)(A)\to
\lambda(D)(A_{01})
\times_{\lambda(D)(A_1)}\ldots\times_{\lambda(D)(A_{n-1})}
\lambda(D)(A_{n-1,n})
$$ is an equivalence.
This easily follows from the Segal properties (D1) and (D3) for $D$
formulated in~\ref{sss:dop}.

By Lemma~\ref{lem:completeness-fiber} applied to $\Fin_*$ viewed as a
complete Segal space, completeness of 
$\lambda(D)$ means that for any $I_*\in\Fin_*$ 
and the map $\iota_I:\Delta\to\Delta_{/\Fin_*}$ 
carrying $[n]\in\Delta$ to $[n]\to[0]\stackrel{I_*}{\to}\Fin_*$, the map
$\iota_I^*:P(\Delta_{/\Fin_*})\to P(\Delta)$ carries $\lambda(D)$
to a complete Segal space. Denote $D_1=\iota^*(D)$, where
$\iota:\Delta\to\Phi$ is defined in~\ref{sss:dop}. This is the
complete Segal space representing the category underlying $D\in\DOp$.
Since $\iota_I^*(\lambda(D))=D_1^I$, it is a complete Segal space.
 Thus, $\lambda(D)$ is a category over 
$\Fin_*$. We
will denote it explicitly by $p:\lambda(D)\to \Fin_*$.

Let us now verify that $p:\lambda(D)\to\Fin_*$ is fibrous. The fiber of 
$p$ at $I_*$ is $\iota_I^*(\lambda(D))=D_1^I$.  

(Fib1) 
Given $\alpha:\langle m\rangle\to\langle n\rangle$ inert,
the base change $\lambda(D)_\alpha:=[1]\times_{\Fin_*}\lambda(D)$
is a category over $[1]$ with fibers $D_1^m$ and $D_1^n$ at $0$ and
$1$ respectively. This is obviously a cocartesian fibration classified
by the projection $p_\alpha:D_1^m\to D_1^n$ determined by the inert 
$\alpha$.
Therefore, $\alpha$ has a locally cocartesian lifting $a:x\to p_\alpha(x)$
for each object $x\in D_1^m$. It is now easy to verify the condition of 
Lemma~\ref{sss:cocartesian} that shows that any such $a$ is in fact
cocartesian.

(Fib2) The inert maps $\rho^i:\langle n\rangle\to\langle 1\rangle$
give rise to an equivalence $\lambda(D)_n\to\prod\lambda(D)_1$. This 
is straightforward.

(Fib3) 
It remains to verify the last property of fibrous objects. Fix
$A:[1]\to\Fin_*$ defined by an arrow 
$f:\langle m\rangle\to\langle n\rangle$.
 Given $x\in\lambda(D)_m$ and $y\in\lambda(D)_n$, the map space 
$\Map^f_{\lambda(D)}(x,y)$
can be expressed as the fiber of the natural map
$$
\lambda(D)(A)=D(\omega(A))\to\lambda(D)(m)\times\lambda(D)(n)
$$
at $(x,y)$.
Applying the axiom (D3) to the forest $\omega(A)$, we deduce the
required decomposition 
$$
\Map^f_{\lambda(D)}(x,y)\to\prod_i\Map^{\rho^i\circ f}_{\lambda(D)}(x,\rho^i(y)).
$$

\subsubsection{}
\label{sss:psionarrows}

$\LOp$ is not a full subcategory of $P(\Delta_{/\Fin_*})$. This means that we have to verify that, given a map $f:D\to D'$ with $D,D'\in\DOp$,
the induced map $\omega^*(f):\omega^*(D)\to\omega^*(D')$ preserves 
the inerts. This immediately follows from the description of inerts given above: if $f:\langle m\rangle\to\langle n\rangle$ is inert
and if $f:D\to D' $ is a map, it induces a commutative square
$$
\xymatrix{
&D_1^m\ar[r]\ar[d] &D_1^n\ar[d]\\
&D_1^{\prime m}\ar[r] &D_1^{\prime n}
}
$$
with the vertical arrows induced by $f$ and the horizontal arrows
being the projections determined by $f$.

\subsection{An equivalence $\lambda\circ\delta\to\id$}
\label{ss:lambdadelta}
$ $

In this subsection we construct an equivalence of functors
$\beta:\lambda\circ\delta\to\id$. 
The construction uses, for any $A\in\bF$, the 
 canonical section $s_A:A\to j(\omega(A))$
in $P(\bF)$, where $j:\Phi\to P(\bF)$ is the composition 
$\Phi\stackrel{i}{\to}\LOp\stackrel{}{\to}P(\bF)$, see diagram
(\ref{eq:thediagram}).
In more detail, for an operad
 $P$ in sets the corresponding object $\ell(P)$ in $\LOp$ can be viewed as a presheaf on $\F$
 via the embedding $\LOp\hookrightarrow P(\F)$. By the description
given in~\ref{ss:lop}, the value of
this presheaf  at $A\in\F$
is precisely the set of operad maps $o(\omega(A))\to P$.
This yields, for $P=o(\omega(A))$, a canonical section $s_A:A\to j(\omega(A))$.

We will deduce that $\beta$ is an equivalence from the following result
to be proven in \ref{ss:proof-op-eq}.
 
\begin{prp}
\label{prp:op-eq}
For  $A\in\bF$ and $L\in\LOp$ the canonical section
\begin{equation}
\label{eq:beta-ind}
s_A:A\to j(\omega(A))
\end{equation}
in $P(\bF)$
induces an equivalence
$$
\Map_\LOp(i\circ\omega(A),L)\to\Map_{P(\bF)}(A,L).
$$
\end{prp}
 
\subsubsection{}
\label{sss:342}
Just for now, let us write $g:\LOp\hookrightarrow P(\F)$ for the embedding functor.
We will first define a morphism of functors
$\beta':g\circ\lambda\circ\delta\to g$ from $\LOp$ to 
$P(\bF)$, 
and then will show that $\beta'$ factors through 
a $\beta:\lambda\circ\delta\to\id$.

Using the standard equivalence
$$
\Fun(A,\Fun(B,C))=\Fun(A\times B,C),
$$
we will define instead an equivalence 
$\wt\beta':\wt{g\circ\lambda\circ\delta}\to\wt{g}$ of functors from 
$\LOp\times\bF^\op$ to $\cS$.
The functor $\wt{g\circ\lambda\circ\delta}$ carries
$(L,A)\in\LOp\times\bF^\op$ to
$$\Map_\DOp(\omega(A),\delta(L))=\Map_\LOp(i\circ\omega(A),L)
\subset\Map_{P(\bF)}(j\circ\omega(A),g(L)),$$
whereas $\wt g$ carries $(L,A)$ to 
$\Map_{P(\bF)}(A,g(L))$.

The functor $\wt\beta'$ is now defined as the precomposition
with $s_A:A\to j\circ\omega(A)$. According to Proposition~\ref{prp:op-eq},
$\wt\beta'$, and, therefore, $\beta'$, is an equivalence.

Since both $g\circ\lambda\circ\delta(L)$ and $g(L)$ belong to
$\LOp\subset P(\bF)$, the natural equivalence 
$\beta'_L:g\lambda\delta(L)\to g(L)$ between 
them also belongs to $\LOp$, hence is the image under $g$ of a unique equivalence $\beta_L:\lambda\delta(L)\to L$.
(Note that the inclusion $\LOp\to P(\bF)$ is fully faithful on equivalences since equivalences automatically preserve cocartesian liftings of inerts.)

\subsubsection{}
Note, for further application, the following consequence of
Proposition~\ref{prp:op-eq} which relates two realizations of a forest
as an operad, one in $\DOp$ and the other in $\LOp$.
Define a morphism of functors $\theta:\lambda|_\Phi\to i$ from $\Phi$
to $\LOp$ so that its composition with $g:\LOp\to P(\bF)$
is given by the natural transformation of functors 
$\Phi\times\bF^\op\to\cS$ defined as in \ref{sss:342},
\begin{eqnarray}
\nonumber\Map_{P(\bF)}(A,\omega^*(F))= \Map_\Phi(\omega(A),F)=\Map_\LOp(i\circ\omega(A),i(F))\to \\
\nonumber\Map_{P(\bF)}(A,j(F)),
\end{eqnarray}
where $A\in\bF$ and $F\in\Phi$.

\begin{prp}
\label{prp:psiisi}
The morphism of functors
$\theta:\lambda|_\Phi\to i$
defined above, from the restriction of $\lambda:\DOp\to\LOp$ to $\Phi\hookrightarrow\DOp$ into 
$i:\Phi\hookrightarrow\LOp$,
is an equivalence.
\end{prp}
\qed

\subsection{An equivalence $\id\to\delta\circ\lambda$}
\label{ss:phipsi}
$  $

In this subsection we construct an equivalence of functors 
$\alpha:\id\to\delta\circ\lambda$. This will complete the proof of the 
equivalence of $\LOp$ with $\DOp$.

Let us temporarily write $G:\DOp\to P(\Phi)$ for the embedding.
Since this embedding is fully faithful, it is sufficient to construct an equivalence
$\alpha':G\to G\circ\delta\circ\lambda$ of functors
from $\DOp$ to $P(\Phi)$.
As in \ref{ss:lambdadelta}, we will construct instead an 
equivalence of functors 
$$
\wt\alpha':\wt G\to\wt{G\circ\delta\circ\lambda}
$$
from $\DOp\times\Phi^\op$ to $\cS$. The functor $\wt G$ carries
$(D,F)\in\DOp\times\Phi^\op$ to $\Map_\DOp(F,D)$ whereas
$\wt{G\circ\delta\circ\lambda}$ carries $(D,F)$ to 
$\Map_\LOp(i(F),\lambda(D))=\Map_\LOp(\lambda(F),\lambda(D))$, the last 
equivalence following from \ref{prp:psiisi}.

We define the morphism $\wt\alpha'$ simply as the morphism
\begin{equation}
\label{eq:theta}
\Map_\DOp(F,D)\to\Map_\LOp(\lambda(F),\lambda(D))
\end{equation}
induced by $\lambda$.  It remains to verify that (\ref{eq:theta}) is an equivalence. By Lemma~\ref{lem:retract} we can choose $A\in\bF$ so that $F$ is a retract of 
$\omega(A)$. Then the composition
\begin{eqnarray}
\nonumber\Map_\DOp(\omega(A),D)\to\Map_\LOp(\lambda(\omega(A)),\lambda(D))=\Map_\LOp(i\circ\omega(A),\lambda(D))\to\\
\nonumber\Map_{P(\bF)}(A,\omega^*(D)),
\end{eqnarray}
is an equivalence. The last map in the composition is also an equivalence by \ref{prp:op-eq}, so $\wt\alpha'$ is an equivalence
for $\omega(A)$, and, therefore, for $F$.

\section{Operadic algebras}
\label{sec:opalgebras}

\subsection{Reconstruction}

Recall the category of Lurie operads $\LOp$ is a Bousfield localization  of $\Cat^+_{/\Fin^\natural_*}$. The latter category is 
$\Cat$-enriched with the category of functors from $X$ to $Y$ defined
by the formula
$$
\Map_\Cat(K,\Fun^\natural(X,Y))=\Map_{\Cat^+_{/\Fin^\natural_*}}(X\times K^\flat,Y).
$$

This $\Cat$-enrichment is used in the definition of the category of operad algebras: given a pair $\cP,\cQ\in\LOp$, the category 
of $\cP$-algebras in $Q$, $\Alg_\cP(\cQ)$, is defined as 
$\Fun^\natural(\cP,\cQ)$. 
In this subsection we prove that
a Lurie operad $\cP\in\LOp$ can be reconstructed from the category
of $\cP$-algebras in $\cS$ (which is a symmetric monoidal category and therefore can be considered as an object 
in $\LOp$). More precisely, one has the following.

\begin{thm}
\label{thm:reconstruction}
Let $f:\cP\to\cQ$ be a morphism of operads which is essentially surjective on colors. Assume that the functor
$$
f^*:\Alg_\cQ(\cS)\to\Alg_\cP(\cS)
$$
is an equivalence. Then $f$ is an equivalence of operads.
\end{thm}

Note that the essential surjectivity condition cannot be dropped: the
embedding of a category into its Karoubian envelope induces an equivalence
of the categories of presheaves!
The proof of the theorem is given in \ref{sss:proof-reco}.

Note the following easy result.

\begin{lem}
\label{lem:opeq}
Let $\alpha:X\to Y$ be an operadic equivalence in 
$\Cat^+_{/\Fin^\natural_*}$. Then the map
$$
\Fun^\natural(Y,\cS)\to\Fun^\natural(X,\cS),
$$
where $\cS$ is considered as a Lurie operad, is an equivalence.
\end{lem}
\begin{proof}
Given $K\in\Cat$, the category $\cS^K=\Fun(K,\cS)$ has a cartesian symmetric monoidal structure, so it can be considered as an object of
$\LOp$. The operadic equivalence $\alpha:X\to Y$ induces
an equivalence
$$
\Map_{\Cat^+_{/\Fin^\natural_*}}(Y,\cS^K)\to
\Map_{\Cat^+_{/\Fin^\natural_*}}(X,\cS^K).
$$
Now the equivalence
$$
\Map_{\Cat^+_{/\Fin^\natural_*}}(X,\cS^K)=
\Map_{\Cat^+_{/\Fin^\natural_*}}(X\times K^\flat,\cS)
$$
yields an equivalence $\Map(K,\Fun^\natural(Y,\cS))=\Map(K,\Fun^\natural(X,\cS))$.
\end{proof}

\begin{rem}
Although Lemma~\ref{lem:???} is sufficient for our purposes, the following more general result can be proven in the same way.
Let $\cO$ be an arbitrary Lurie operad. Using a full embedding
of $\cO$ into a symmetric monoidal category $\hat\cO$, see~\ref{???} below, 
an operadic equivalence  $\alpha:X\to Y$ gives rise to an equivalence
$$
\Fun^\natural(Y,\cO)\to\Fun^\natural(X,\cO)
$$
for any $\cO$. Indeed, the induced symmetric monoidal structure on $\hat\cO^K$ defines an operad structure on the full subcategory
$\cO^K:=\Fun(K,\cO)\times_{\Fun(K,\Fin_*)}\Fin_*$ as in \ref{???}.
\end{rem}

\subsubsection{}

Let $I$ be a set and let $\cP$ be a Lurie operad.
A map $r:I\to\cP_1$ is called {\sl a recoloring} if it induces a surjective map on the equivalence classes of objects of $\cP_1$.
We define a {\sl recolored operad} as an operad $\cP$ endowed
with a recoloring $r:I\to\cP_1$.
Any map $r:I\to\cP_1$ defines a forgetful functor
$$G_r:\Alg_\cP(\cS)\to\cS^I.$$

A general theorem \cite{L.HA}, 3.1.3.5 implies that $G_r$ admits
a left adjoint functor of free $\cP$-algebra denoted 
$F_r:\cS^I\to\Alg_\cP(\cS)$.

We present below an explicit expression for the free algebra $F_r(X)$
where $p:X\to I$ is a map of sets, considered as a collection of 
(discrete) spaces 
$X_i=p^{-1}(i)\in\cS$. This is the free $\cP$-algebra generated by the 
set $X$ of objects such that the color of $x\in X$ is $r(p(x))$. 
Note that $\cP$-algebras with values in 
$\cS$ can be described by functors $A:\cP\to\cS$ that are {\sl monoid objects} in the sense of \cite{L.HA}, 2.4.2.1. 
Equivalently, this means that the left fibration $\cP_A\to\cP$ classified by $A$ is a left 
fibration of operads.

Following \cite{L.HA}, 2.1.1.20, we denote by $\Triv\subset\Fin_*$ the subcategory spanned by the inert arrows. This is the trivial operad on one color. 
For a given set $X$ we denote by $\Triv_X$ the coproduct of $X$ copies of the operad $\Triv$, so that $\Triv_X$ is the trivial operad on $X$ colors. 
The objects of $\Triv_X$ are finite sets over $X$ and the arrows are embeddings of these sets, considered as (inert) arrows in the opposite direction. 

The map $c:=r\circ p:X\to\cP_1$ extends to $c:\Triv_X\to\cP$, see
\cite{L.HA}, 2.1.3.6, and, therefore, gives rise to the functor
\begin{equation}
\bar c:\Triv_X^\op\to\Left(\cP)=\Fun(\cP,\cS)
\end{equation}
with values in the category of left fibrations over $\cP$, carrying 
$\alpha:U\to X$ in $\Triv_X$
to the left fibration $\cP_{c\alpha/}\to\cP$. We finally denote
\begin{equation}
\label{eq:Frx}
F_r(X)=\colim(\bar c)\in\Left(\cP).
\end{equation}
Rewriting (\ref{eq:Frx}) as a functor $F_r(X):\cP\to\cS$, we deduce the formula
\begin{equation}
\label{eq:Frc-2}
F_r(X)(d)=\colim(\bar c_d)
\end{equation}
where $\bar c_d:\Triv^\op_X\to\cS$ is the functor
carrying $\alpha\in\Triv_X$ to $\Map_\cP(c\circ\alpha,d)$.

This can be further rewritten as follows. Let $D_*$ be the image of $d\in\cP$ in $\Fin_*$. Define the category $\Triv_{X,D}$ whose objects
are the pairs $(\alpha:U\to X, \beta:U_*\to D_*)$ and morphisms
$(\alpha,\beta)\to(\alpha':U'\to X,\beta':U'_*\to D_*)$ defined by an 
embedding $U'\to U$ over $X$ so that the corresponding inert arrow
$U_*\to U'_*$ commutes with the $\beta$'s. One has an obvious forgetful functor $\phi:\Triv_{X,D}\to\Triv_X$ and a functor
$$
\bar c_{d,D}:\Triv^\op_{X,D}\to\cS
$$
carrying $(\alpha,\beta)$ to $\Map^\beta_\cP(c\circ\alpha,d)$.
The fibers of $\phi$ are discrete, so obviously $\bar c_d$ is a left
Kan extension of $\bar c_{d,D}$ along $\phi$. Therefore,
$F_r(X)(d)=\colim\bar c_d=\colim\bar c_{d,D}$.
The category $\Triv_{X,D}$ has a subcategory $\Triv^\act_{X,D}$
spanned by the pairs $(\alpha,\beta)$ with $\beta$ active. This is a groupoid. We denote by $\bar c_{d,D}^\act$ the restriction of
$\bar c_{d,D}$ to $(\Triv^\act_{X,D})^\op$. The embedding $\Triv^\act_{X,D}\to\Triv_{X,D}$ is cofinal
so that it induces an equivalence of colimits 
$$
\colim\bar c^\act_{d,D}\to\colim\bar c_{d,D}.
$$
We can finally reformulate the description of 
$F_r(X)(d)=\colim\bar c^\act_{d,D}$ as follows.
Let $\Triv_X^\eq$ be the maximal subgroupoid of $\Triv_X$
(this is just the groupoid of finite sets over $X$) and let $\Map_\cP^\act(x,y)$ denote
the space of active arrows in $\cP$ from $x$ to $y$. 
The forgetful functor $\Triv^\act_{X,D}\to\Triv^\eq_X$ having discrete fibers, the left Kan extension of $\bar c^\act_{d,D}$ along it
yields the functor
$$
\bar c_d^\act:(\Triv_X^\eq)^\op\to\cS
$$
assigning to $\alpha$ the space $\Map_\cP^\act(c\circ\alpha,d)$. We see that
\begin{equation}
\label{eq:Frc-3}
F_r(X)(d)=\colim(\bar c_d^\act).
\end{equation}
We need yet another version of the above formula.

The functor $\bar c_d^\act$ factors through
$\bar p:\Triv^\eq_X\to\Triv^\eq_I$
carrying $\alpha:U\to X$ to $p\circ\alpha$. Therefore, 
 the colimit of
$\bar c_d^\act$ can be rewritten as $\colim\bar X$
where $\bar X$ is the left Kan extension of $\bar c_d^\act$
with respect to $\bar p$. One easily sees that
$\bar X:(\Triv^\eq_I)^\op\to\cS$ is defined by the formula
\begin{equation}
\label{eq:barX}
\bar X(\gamma)=\Map^\act_\cP(r\circ\gamma,d)\times_{\Aut_I(U)}\Hom_I(U,X).
\end{equation}
for $\gamma:U\to I$.  
In the special case $d\in\cP_1$ this can be
rewritten as
\begin{equation}
\label{eq:barX-2}
\bar X(\gamma)=\cP(r\circ\gamma,d)\times_{\Aut_I(U)}\Hom_I(U,X).
\end{equation}

\begin{prp}
$F_r(X)$ is a free $\cP$-algebra generated by the set $X$.
\end{prp}
\begin{proof}
Let $q:\cQ\to\cP$ be a left fibration of operads. One has
$$
\Map_{\Cat_\cP}(F_r(X),\cQ)=\lim_{\alpha\in\Triv^\op_X}
\Map_{\Cat_\cP}(\cP_{c\circ\alpha/},\cQ)=
\lim_{\alpha\in\Triv^\op_X}\cQ_{c\circ\alpha}=\Map_{\Cat_\cP}(X,\cQ).
$$
\end{proof}

\subsubsection{Proof of \ref{thm:reconstruction}}
\label{sss:proof-reco}

Choose a  recoloring $r:I\to\cP_1$. It will automatically give a recoloring $f\circ r:I\to\cQ$.
This yields a commutative diagram
\begin{equation}
\label{eq:AlgPQ}
\xymatrix{
&\Alg_\cP\ar_{G_r}[rd]& &\Alg_\cQ\ar^{f^*}[ll]\ar^{G_{f\circ r}}[dl]\\
&&\cS^I&
}
\end{equation}

We denote by $f_!:\Alg_\cP\to\Alg_\cQ$ the functor left adjoint (and inverse)
to $f^*$, so we get an equivalence
$$F_{f\circ r}=f_!\circ F_r.$$
This yields an equivalence
\begin{equation}
\label{eq:freeeq}
G_r\circ F_r\to G_{f\circ r}\circ F_{f\circ r}.
\end{equation}
The source and the target of the above map are explicitly given as colimits, see formulas (\ref{eq:Frc-3}) and (\ref{eq:barX-2}). Thus, to yield an equivalence (\ref{eq:freeeq}), one should have, for any   $\gamma:U\to I$
in $\Triv^\eq_I=\Fin^\eq_{/I}$ and a map of sets $X\to I$, an equivalence
$$
\cP(r\circ\gamma,d)\times_{\Aut_I(U)}\Hom_I(U,X)\to
\cQ(r\circ\gamma,d)\times_{\Aut_I(U)}\Hom_I(U,X).
$$

Choosing $X\to I$ large enough for $\Aut_I(U)$ to have an orbit in $\Hom_I(U,X)$ with trivial stabilizer (for example, choosing $X\to I$ to be $U\to I$ itself),
 we deduce that the map
$\cP(r\circ\gamma,d)\to\cQ(r\circ\gamma,d)$ has to be an equivalence for 
all $\gamma:U\to I$ and $d\in\cP_1$. 
This implies that $f:\cP\to\cQ$ is an equivalence.
\qed

\subsection{Model structures on operad algebras}
\label{ss:models}
In this subsection we present standard results on model structures in categories of algebras and rectification results.

\subsubsection{}
\label{sss:PSrectification}
We will use a special case of the rectification theorem of 
Pavlov-Scholbach.
Let $\cC$ be a simplicial symmetric monoidal model category.
According to \cite{NS}, A.7, the underlying $\infty$-category 
$\cC_\infty$ inherits a symmetric monoidal structure so that the localization functor $\cC\to\cC_\infty$ is lax symmetric monoidal.
 
Let  $\cO$ be a  $\Sigma$-free and 
$\cC$-admissible operad in sets.~\footnote{Pavlov and Scholbach more generally consider simplicial operads.}
Then the category of algebras $\Alg_\cO(\cC)$ has a projective model 
structure.
One has a functor
$L':\Alg_\cO(\cC)^{cf}\to\Alg_{\ell(\cO)}(\cC_\infty)$
%assigning to a map of operads $\cO\stackrel{A}{\to}\cC$ its composition
%with the localization functor
%$\ell(\cO)\stackrel{\ell(A)}{\to}\ell(\cC)\stackrel{}{\to}\cC_\infty$. The functor  $L'$ carries 
carrying weak equivalences of fibrant cofibrant algebras  to equivalences, hence
inducing a
functor
\begin{equation}
\label{eq:Phi}
L:\Alg_\cO(\cC)_\infty\to\Alg_{\ell(\cO)}(\cC_\infty)
\end{equation}
between the underlying $\infty$-categories.  

\begin{thm}[see~\cite{PS}, Theorem 7.11]
\label{thm:rect}
Let $\cC$ be a simplicial symmetric monoidal model category and let $\cO$ be a 
$\Sigma$-free $\cC$-admissible   operad as above. Then the functor 
$L$ in (\ref{eq:Phi}) is an equivalence.
\end{thm}
\qed

The following Lemma~\ref{lem:correctfp} is used as an inductive step in the proofs of Proposition~\ref{prp:op-eq} (see~\ref{ss:proof-op-eq}) and of Lemma~\ref{thm:BVL}.

\begin{lem}
\label{lem:correctfp} Let $A:[1]\to\Fin_*$ be presented by an
arrow $f:I_*\to J_*$ and let $\cO=o(\omega(A))$.
Let $\cC$ be a simplicial model category. We endow $\Alg_\cO(\cC)$ with the projective model structure. Then the forgetful functor
$\Alg_\cO(\cC)\to\cC^I$
defined by the source of $f$
induces a fibration of the simplicial categories of fibrant cofibrant
objects 
$$
p:\Alg_\cO(\cC)^{cf}_*\to(\cC^{cf}_*)^I.
$$
\end{lem}

\begin{proof}
The forest $\omega(A)$ consists of corollas numbered by $j\in J$
and trivial operads $\eta$ numbered by $f^{-1}(*)\setminus\{*\}\subset I$. The claim 
immediately reduces to the case when $\omega(A)$ is a single corolla
$C_n$. Thus, from now on we assume $\cO=o(C_n)$.

A $C_n$-algebra in $\cC$ is given by an arrow 
$\alpha:D_1\times\ldots \times D_n\to D_0$ in $\cC$. It is a fibrant 
cofibrant object if $D_i$ are fibrant cofibrant and $\alpha$ is a 
cofibration. Given two such objects,
$\alpha$ as above and $\beta:E_1\times\ldots\times E_n\to E_0$, the simplicial 
set $\shom(\alpha,\beta)$ is defined as the fiber product
$$
\prod_{i=1}^n\shom_\cC(D_i,E_i)\times_{\shom(\prod_{i=1}^n D_i,E_0)}\shom(D_0,E_0).
$$
The map $\shom(\alpha,\beta)\to\prod_{i=1}^n\shom_\cC(D_i,E_i)$ is a fibration because it is obtained by base change from
the map 
$
\shom_\cC(D_0,E_0)\to\shom(\prod_{i=1}^n(D_i,E_0),
$
which is itself a fibration because it is defined  by the composition with the cofibration $\alpha$. 
It remains to verify that the induced map of the homotopy categories
$$
\Ho(\Alg_{C_n}(\cC)^{cf}_*)\to\Ho(\cC_*^{cf})^n
$$
is an isofibration of conventional categories. This is straightforward.
\end{proof}

\subsection{Proof of~\ref{prp:op-eq}}
\label{ss:proof-op-eq}
\subsubsection{}
The map $s_A:A\to j(\omega(A))$ introduced at the beginning of Section~\ref{ss:lambdadelta},
induces a map $s'_A:A'\to i(\omega(A))$  in $\LOp$ 
where $A^\flat\to A'$ is an operadic equivalence in 
$\Cat^+_{/\Fin^\natural_*}$; see 
\ref{sss:lop-bousfield} for the notion of operadic equivalence.

By the reconstruction theorem~\ref{thm:reconstruction} it is sufficient to
verify that $s'_A$ induces an equivalence of the categories of algebras
with values in $\cS$,
$$
s^{\prime*}_A:\Alg_{i(\omega(A))}(\cS)\to\Alg_{A'}(\cS)
=\Fun_{\Cat_{/\Fin_*}}(A,\cS),
$$
where $A$ and $\cS$ on the right-hand side of the formula
are considered as  objects of $\Cat_{/\Fin_*}$.
(the equality in the last formula follows from the operadic 
equivalence $A^\flat\to A'$ and Lemma~\ref{lem:opeq}).

We will prove that $s^{\prime*}_A$ is an equivalence by induction,
based on Lemma~\ref{lem:correctfp}.

\subsubsection{Pruning a simplex}
\label{sss:pruning}
The following procedure
of pruning a simplex $A$ will be used.
Define $B:[n-1]\to\Fin_*$ and $C:[1]\to\Fin_*$ by the formulas
$B=A\circ d_n$, $C=A\circ d_0^{n-1}$. Let $v:[0]\to\Fin_*$ be
defined by $v=C\circ d_1$. The map $v$ is given by an object $V_*\in\Fin_*$. The decomposition $A=B\sqcup^vC$ in $\Cat_{/\Fin_*}$
gives rise to a commutative diagram
\begin{equation}
\label{eq: pruning-diag}
\xymatrix{
&{\Alg_{i(\omega(A))}(\cS)}\ar[r]\ar[d]
&{\Alg_{i(\omega(A))}(\cS)\times_{\cS^V}\Alg_{i(\omega(A))}(\cS)}\ar[d]\\
&{\Fun_{\Fin_*}(A,\cS)}\ar[r]&{\Fun_{\Fin_*}(B,\cS)\times_{\cS^V}
\Fun_{\Fin_*}(C,\cS)}
}
\end{equation}
so that the lower horizontal arrow is, obviously, an equivalence.
In \ref{lem:pruning-dec} we will verify that the upper horizontal arrow is also an equivalence.
This will reduce the claim of~\ref{prp:op-eq} that $s^{\prime*}_A$ 
to the case $n=1$ which is very easy.

\begin{lem}
\label{lem:pruning-dec}
The map
$
\Alg_{i(\omega(A))}(\cS)\to
\Alg_{i(\omega(A))}(\cS)\times_{\cS^V}\Alg_{i(\omega(A))}(\cS)
$
defined by the decomposition $A=B\sqcup^vC$ is an equivalence.
\end{lem}
\begin{proof}

Clearly, $o(\omega(A))=o(\omega(B))\sqcup^{\Triv_V} o(\omega(C))$ where 
$\Triv_V$ is the trivial operad on $V$ colors.  
The operad $o(\omega(A))$ is free as an operad in sets. The category
$\Alg_{i(\omega(A))}(\cS)$ is the $\infty$-category underlying
the simplicial model category $\Alg_{o(\omega(A))}(\sSet)$, where
the model structure is the projective model structure induced from 
the standard model structure on the simplicial sets.

The category $\Alg_{o(\omega(A))}(\sSet)$ is equivalent to
the fiber product
\begin{equation}
\label{eq:fiberproduct-mod}
\Alg_{o(\omega(B))}(\sSet)\times_{(\sSet)^V}
\Alg_{o(\omega(C))}(\sSet).
\end{equation}
Moreover, an arrow $f:X\to Y$ in $\Alg_{o(\omega(A))}(\sSet)$
is a fibration, cofibration or weak equivalence iff its components
satisfy the same property in the corresponding model categories
$\Alg_{o(\omega(B))}(\sSet)$ and $\Alg_{o(\omega(C))}(\sSet)$.

Note that one needs to be careful as the fiber product in  (\ref{eq:fiberproduct-mod}) is taken in the category of conventional categories, and not in $\Cat$. However, the same fiber product formula still holds for the underlying $\infty$-categories. Indeed, 
the $\infty$-category underlying $\Alg_{o(\omega(A))}(\sSet)$ is the
homotopy coherent nerve of the simplicial category of fibrant cofibrant 
objects which is the fiber product of the simplicial categories of 
fibrant cofibrant objects of $\Alg_{o(\omega(B))}(\sSet)$
and $\Alg_{o(\omega(C))}(\sSet)$.
By Lemma~\ref{lem:correctfp} this
fiber product calculates the fiber product of the corresponding 
$\infty$-categories in $\Cat$.

This completes the proof of Proposition~\ref{prp:op-eq}.
\end{proof}

The following result deals with a slightly different type of 
pruning; it was used in~\ref{ss:inDOp} and its proof
is very similar to that of \ref{lem:pruning-dec}.

\begin{prp}
\label{prp:cuttingtree}
Let $b$ be an inner edge of a tree $T$. Then $i(T)\in\LOp$ is a pushout
$$
i(T)=i(T^b)\sqcup^{i(b)}i(T_b).
$$
\end{prp}
\begin{proof}
 By Theorem~\ref{thm:reconstruction}, the claim reduces to proving that
the natural map 
$$
\Alg_{i(T)}(\cS)\to\Alg_{i(T^b)}(\cS)\times_{\Alg_{i(b)}(\cS)}\Alg_{i(T_b)}(\cS)
$$
is an equivalence of $\infty$-categories. By~\cite{PS}, 
see~\ref{thm:rect}, the 
$\infty$-categories of algebras involved underly the simplicial model 
categories of algebras with values in $\sSet$. The category 
$\Alg_{o(T)}(\sSet)$ is equivalent to the fiber product
$$
\Alg_{o(T^b)}(\sSet)\times_{\sSet}\Alg_{o(T_b)}(\sSet)
$$
and the reasoning of \ref{lem:pruning-dec} based on Lemma~\ref{lem:correctfp} proves that
this equivalence induces an equivalence of the underlying $\infty$-categories.
\end{proof}

\section{Monoidal structures}
\label{sec:monoidal}
\subsection{}
In Section~\ref{ss:dop} we introduced the category $\DOp$ underlying a Quillen model 
category of simplicial presheaves. It is known~\cite{HM} that the 
associated homotopy category carries a structure of symmetric monoidal category. Our goal in this section is to explain that this structure can
be lifted to the structure of a symmetric monoidal $\infty$-category
on $\DOp$, and to prove the following sharpening of
Theorem~\ref{thm:equivalence}.

\begin{thm}
\label{thm:SMequivalence}
The functor $\lambda:\DOp\to\LOp$ is an equivalence of symmetric monoidal categories.
\end{thm}
It follows from this theorem that its inverse $\delta:\LOp\to\DOp$
is symmetric monoidal as well. Even though we already know that 
$\lambda$ and $\delta$ form an equivalence of categories, the proof of this stronger theorem is quite involved, due to the fact that 
(especially in the $\infty$-context!) it is difficult to deal with the rich structure of a symmetric monoidal category in a direct way.

\subsection{Preliminaries}

\subsubsection{} We consider symmetric monoidal categories and operads
in the sense of Lurie~\cite{L.HA}, so ``operad'' means object of 
$\LOp$ here. 

It is convenient to define symmetric monoidal categories as commutative algebras in $\Cat$, that is the functors $\Fin_*\to\Cat$ satisfying
the Segal condition. The (covariant) Grothendieck construction then 
realizes the category $\Cat^\SM$ of symmetric monoidal categories
as the subcategory of $\LOp$ whose objects are cocartesian fibrations
of operads $p:M\to\Fin_*$, with the morphisms preserving cocartesian 
arrows. It is convenient to have another realization, the one connected
to the contravariant Grothendieck construction. The following 
terminology is taken from ~\cite{BGS}. 
\begin{dfn}
A functor $q:\cC\to\Fin_*^\op$ is called  {\sl an anti-operad} if $q^\op:\cC^\op\to\Fin_*$ is fibrous (that is, a Lurie operad).
\end{dfn}

The category of anti-operads will be denoted by $\Coop$. The contravariant realization of $\Cat^\SM$ identifies $\Cat^\SM$  with the subcategory of $\Coop$
 whose objects are cartesian fibrations and whose arrows preserve the cartesian arrows. The categories of operads and of anti-operads are
obviously equivalent. However, if $M$ is a conventional symmetric monoidal category, its operadic realization $M^\otimes$ assigns to $X_1,\ldots,X_n$ and $Y$ in $M$ the set $\Hom_M(\otimes X_i,Y)$ of operations, whereas 
its anti-operadic realization $^\otimes\!M$ assigns the set $\Hom_M(Y,\otimes X_i)$
of ``anti-operations''.
The passage to the opposite symmetric monoidal category intertwines between the two realizations: $(^\otimes\! M)^\op=M^{\op\otimes}$.
For an operad $L$ we denote by $\hat L$ the symmetric monoidal
envelope of $L$. Passing to opposite categories, we define the 
enveloping symmetric monoidal category $\hat C$ of an anti-operad $C$. One has
canonical embeddings $L\to\hat L^\otimes$ and $C\to^\otimes\hat C$
so that if $C=L^\op$, $\hat C=\hat L^\op$.

We will now define two notions intermediate between the world of (anti) operads and the world of symmetric monoidal categories.

\begin{dfn}
\label{dfn:laxSM}
An operad $p:\cO\to\Fin_*$ is called {\sl a lax symmetric monoidal category} if $p$ is a locally cocartesian fibration, see~\cite{L.T},
2.4.2.6.
\end{dfn}

\begin{dfn}
\label{dfn:colaxSM}
\begin{itemize}
\item[1.]
An anti-operad $q:\cC\to\Fin_*^\op$ is called {\sl a colax symmetric monoidal category} if $q$ is a locally cartesian fibration.
\item[2.] $q:\cC\to\Fin_*^\op$ is called {\sl a colax symmetric monoidal category with colimits} if, in addition to the above, the fiber $\cC_1$
has colimits and the maps 
$\otimes_n:\cC_1^n\to\cC_1$ defined by the local cartesian liftings preserve colimits in each argument.
\end{itemize}
\end{dfn}

In the conventional setting, a colax symmetric monoidal category $\cC$ is given by a 
collection of operations $\otimes_n:\cC^n\to\cC$, with a compatible 
collection of natural transformations (not necessarily equivalences) 
of the form
$$
\otimes_m\circ(\otimes_{n_1}\times\ldots\times\otimes_{n_m})\to
\otimes_n
$$
with $n=\sum n_i$.
Note that, since the collections of active and inert arrows in $\Fin_*$ form a factorization system,  it is sufficient to require that the
active arrows have a locally (co)cartesian lifting.

\subsubsection{Day convolution}

The result  \cite{L.HA}, 4.8.1.10  yields the following.
\begin{Lem}
\label{lem:Day-univ}
Let $\cC$ be a symmetric monoidal category. Then the category of presheaves $P(\cC)$ inherits a symmetric monoidal structure, so that
the Yoneda embedding $Y:\cC\to P(\cC)$ is a symmetric monoidal functor, universal
among symmetric monoidal functors from $\cC$ to a symmetric monoidal category with colimits. 
\end{Lem}

\subsubsection{Presheaves on an anti-operad}
\label{sss:presheaves}
We will now define, for any anti-operad $C$, a full embedding
of anti-operads $C\to P$ where $P$ is a colax symmetric monoidal category
with colimits whose underlying category is $P(C_1)$.
 
Let $C$ be an anti-operad and let $\hat C$ be the symmetric monoidal envelope of $C$. We write $C_1$ for the category underlying $C$. Then the full embedding $u:C_1\to(^\otimes\hat C)_1=\hat C$ induces an adjoint pair
$$
u_!:P(C_1)\rlarrows P(\hat C):u^*
$$
where $u_!$ is again a full embedding. In fact, 
let $f=\colim(Y\circ a)$, $f'=\colim(Y\circ a')$
for $a:K\to C_1$ and $a':K'\to C_1$. Then 
$$\Map_{P(C_1)}(f,f')=\lim_{k\in K}\colim_{k'\in K'}\Map(a(k),a'(k')).$$
The same formula describes $\Map_{P(\hat C)}(u_!(f),u_!(f'))$,
so $u_!$ is a full embedding.
This implies that $P(C_1)$, as a full subcategory of a symmetric 
monoidal category, inherits the
structure of an anti-operad from $^\otimes P(\hat C)$. We will denote it 
by $^\otimes P(C_1)$ and we claim that it is a colax symmetric monoidal 
category.
This means that for any $f_1,\ldots,f_n$ in $P(C_1)$ 
the functor  
$$
f\in P(C_1)\mapsto \Map(u_!(f),\otimes_i u_!(f_i))
$$
is representable. This is obviously so as 
$$
\Map(u_!(f),\otimes_i u_!(f_i))=
\Map(f,u^*(\otimes_i u_!(f_i))).
$$
Therefore, the multiple tensor product functor on $P(C_1)$
is defined as the composition
\begin{equation}
\label{eq:laxproduct}
P(C_1)^{\otimes n}\stackrel{u_!^{\otimes n}}{\longrightarrow}
P(\hat C)^{\otimes n}\stackrel{\otimes_n}{\to}
P(\hat C)\stackrel{u^*}{\to}P(C_1).
\end{equation}
Note that by construction the map of anti-operads 
$C\to^\otimes\!\hat C$ factors through the full embedding 
$^\otimes P(C_1)\to^\otimes\! P(\hat C)$ and therefore yields a map 
$C\to^\otimes\! P(C_1)$.
 
\subsection{$\DOp$ as a symmetric monoidal category}
First of all, recall that the conventional category $\Op(\Set)$ of operads in sets is symmetric monoidal. Its tensor product is the Boardman-Vogt tensor product of operads, denoted
$$
P\otimes_{BV}Q
$$
for two operads $P$ and $Q$.  

Next,   the full embedding
$$
\Phi\to\Op(\Set),\ F\mapsto o(F)
$$
gives rise to a full sub-anti-operad $\boldsymbol{\Phi}$ of 
$^\otimes\Op(\Set)$ with $\boldsymbol\Phi_1=\Phi$.
Explicitly,
$$
\boldsymbol\Phi(F; F_1\ldots,F_n)=\Hom_{\Op(\Set)}(o(F),o(F_1)\otimes_{BV}\ldots\otimes_{BV} o(F_n)).
$$
It follows that $P(\Phi)$ has the structure of a colax symmetric monoidal category as explained in~\ref{sss:presheaves}.

An arrow $f$ in $P(\Phi)$ will be called {\sl an operadic equivalence}
if it is carried to equivalence by the localization functor
$P(\Phi)\to\DOp$.

\begin{prp}
\label{prp:DOp-sm}
\begin{itemize} 
\item[1.] Multiple tensor products $\otimes_n:P(\Phi)^n\to P(\Phi)$
preserve operadic equivalences in each argument. 
\item[2.] The localization functor $P(\Phi)\to\DOp$ canonically extends to
a map of colax symmetric monoidal categories.
\item[3.] The localization functor $P(\Phi)\to\DOp$ carries associativity constraints to equivalences. Therefore, the colax symmetric monoidal structure on $\DOp$ is in fact symmetric monoidal.
\end{itemize}
\end{prp}
\begin{proof}
This result easily follows from the properties of shuffles of trees presented in the Appendix.
Recall that $\DOp$ is a Bousfield localization of $P(\Phi)$ with respect to three types of arrows.
\begin{itemize}
\item[1.] $T_d\sqcup^dT^d\to T$, where $d$ is an inner edge of a tree $T$.
\item[2.] $*\to J$, embedding of simplicial sets considered as objects of $P(\Phi)$.
\item[3.] $\sqcup\ T_i\to F$ where $F\in\Phi$ and $T_i$ are the tree components of $F$.
\end{itemize}
To prove the first claim, we have to show that for all 
$f_i\in P(\Phi)$ the functor 
\begin{equation}
\label{eq:tensoreval}
\otimes_n(f_1,\dots,f_{k-1}, -,f_{k+1},\dots,f_n):P(\Phi)\to P(\Phi)
\end{equation}
carries the arrows of types 1--3 to operadic equivalences. Since $\otimes_n$ preserves colimits and the localization
functor $P(\Phi)\to\DOp$ preserves colimits, it is enough to verify this claim in the case when $f_i$ are 
representable, that is, forests. To calculate the tensor product, 
we can replace each forest with the coproduct of its tree components;
in this way the claim reduces to the case when all $f_i$ are trees. 
For the arrows of type 1 the result now follows from Proposition~\ref{prp:app-segal}. The arrow $[0]\to J$ is carried
by (\ref{eq:tensoreval}) to a deformation retract of dendroidal sets, so to an operadic equivalence.
The arrows of type 3 are obviously carried to equivalences.

{\sl Claim 2.} To see that Claim 1 defines a colax symmetric monoidal
structure on $\DOp$, we look at the anti-operadic presentation
$q:^\otimes\!\!P(\Phi)\to\Fin_*^\op$ of $P(\Phi)$. 
Claim 1 implies that the localization of the total category $^\otimes\!P(\Phi)$ with respect to
operadic equivalences yields a locally cartesian fibration 
$^\otimes\DOp\to\Fin_*^\op$; moreover, the localization map preserves locally cartesian arrows.

{\sl Claim 3.} We now look at the morphisms of functors 
\begin{equation}
\otimes_q\circ(\id^k\times\otimes_p\times\id^{q-k-1})\to\otimes_{p+q+1}:
P(\Phi)^{p+q+1}\to P(\Phi)
\end{equation}
describing the associativity constrants. Both source and target preserve colimits on each argument, so the claim is reduced  to the case when $f=\{f_i\}\in P(\Phi)^{p+q+1}$ is a collection of trees. Then Proposition~\ref{prp:app-asso} implies the result.
\end{proof}

\subsection{Proof of Theorem~\ref{thm:SMequivalence}}
Recall \cite{L.HA} that $\LOp$ has the structure of a symmetric monoidal category. For two objects $L$ and $M$ their tensor product is characterized by the property that there is an equivalence
$$
\Alg_{L\otimes M}(\cS)=\Alg_L(\Alg_M(\cS)).
$$
(This characterizes $L\otimes M$ uniquely by the reconstruction 
theorem~\ref{thm:SMequivalence}.) Since $\LOp$ is symmetric monoidal, it has the structure of an anti-operad.
\begin{prp}
\label{prp:i-coop}
The inclusion $i:\Phi\to\LOp$ canonically extends to a map of 
anti-operads.
\end{prp}
\begin{proof}
Recall that $i$ is the composition $\Phi\stackrel{o}{\to}\Op(\Set)
\stackrel{\ell}{\to}\LOp$. Let $\cC$ be the full subcategory of $\LOp$
spanned by the objects $i(F),\ F\in\Phi$. Since $\LOp$ is a symmetric monoidal category, $\cC$ \VH{acquires} \IM{receives} the structure of a anti-suboperad.
We will verify that $\cC$
is a conventional anti-operad canonically isomorphic to 
$\boldsymbol\Phi$.

Given a sequence $O_1,\dots,O_n$ of operads in sets, one has a canonical
operad multifunctor
$$
\ell(O_1)\times\dots\times \ell(O_n)\to \ell(O_1\otimes_{BV}\dots\otimes_{BV}O_n),
$$
expressing the universal property of Boardmann-Vogt tensor product. In particular, a sequence $F_1,\ldots,F_n$ of objects of $\Phi$ yields an operad multifunctor
$$
iF_1\times\ldots\times iF_n\to\ell(o(F_1)\otimes_{BV}\ldots\otimes_{BV}o(F_n))
$$ 
that induces a map of operads (see \ref{sss:smlop})
$$
\theta:iF_1\otimes\ldots\otimes iF_n\to\ell(o(F_1)\otimes_{BV}\ldots\otimes_{BV}o(F_n)).
$$ 
Thus, it suffices to verify that this map is an equivalence in $\LOp$. Indeed, we would then have an equivalence of anti-operads $\Phi$ and 
$\cC$ since for any object $F\in\Phi$ the induced map from
$$
\Hom_{\Op(\Set)}(o(F),o(F_1)\otimes_{BV}\ldots\otimes_{BV}o(F_n))=
\Map_{\LOp}(iF,\ell(o(F_1)\otimes_{BV}\ldots\otimes_{BV}o(F_n)))
$$
to $\Map_{\LOp}(iF,iF_1\otimes\ldots\otimes iF_n)$, will be then an equivalence. The fact that $\theta$ is an equivalence now follows by induction from the following lemma.
\end{proof}

\begin{lem}
\label{thm:BVL}
Let $\cP=o(F)$ where $F$ is a forest
and let $\cQ$ be a $\Sigma$-free operad in $\Set$. Then 
the canonical operad bifunctor $\ell(\cP)\times\ell(\cQ)\to\ell(\cP\otimes_{BV}\cQ)$ exhibits
$\ell(\cP\otimes_{BV}\cQ)$ as a tensor product (in the sense of Lurie)
of $\ell(\cP)$ and $\ell(\cQ)$.
\end{lem}

\begin{proof}
By the reconstruction theorem,  it is sufficient
to verify that the map $\theta:\ell(\cP)\otimes\ell(\cQ)\to\ell(\cP\otimes_{BV}\cQ)$
induces an equivalence of the categories of algebras
$$
\theta^*:\Alg_{\ell(\cP\otimes_{BV}\cQ)}(\cS)\to\Alg_{\ell(\cP)\otimes\ell(\cQ)}(\cS).
$$

By the rectification theorem  the left-hand side
is the $\infty$-category underlying the model category 
\begin{equation}
\label{eq:algsset}
\Alg_{\cP\otimes_{BV}\cQ}(\sSet)=\Alg_\cP(\Alg_\cQ(\sSet)),
\end{equation}
whereas, by definition, the right-hand side is
\begin{equation}
\label{eq:algspaces}
\Alg_{\ell(\cP)}(\Alg_{\ell(\cQ)}(\cS)).
\end{equation}

We denote $\cC=\Alg_\cQ(\sSet)$. This is a simplicial model category 
whose underlying $\infty$-category is $L\cC:=\Alg_{\ell(\cQ)}(\cS)$. We have the 
localization functor $L:\cC\to L\cC$ and we have to verify that
the natural map
$
\Alg_\cP(\cC)\to\Alg_{\ell(\cP)}(L\cC)
$
induces an equivalence 
\begin{equation}
\label{eq:equivalence}
L(\Alg_\cP(\cC))\to\Alg_{\ell(\cP)}(L\cC).
\end{equation}
Recall that  $\cP=o(F)$. We endow $\Alg_\cP(\cC)$
with the projective model structure. Note that we cannot use the 
result of Pavlov-Scholbach~\cite{PS} as $\Alg_\cP(\cC)$ is not a 
monoidal model category.
It is easy to see that a map $f:A\to A'$ in $\Alg_\cP(\cC)$ is a fibration iff its restriction
to any corolla of $F$ is a fibration.
Therefore, the simplicial category $\Alg_\cP(\cC)^{cf}_*$
of fibrant cofibrant algebras is the (naive) fiber product of the
simplicial categories of algebras over the corollas contained in $\cP$. To simply the formulas,
we will proceed by induction on the number of corollas in $F$.

We can write $F=F_1\cup^vF_2$ (pruning/grafting) where $F_2$ is a corolla containing a root of $F$, $F_1$ is a (smaller) forest and $v=\{v_1,\ldots,v_k\}$ is a subset of
the set of leaves of $F_2$ whose elements are identified in $F$ with 
roots of $F_1$.
This decomposition yields a decomposition of operads
$\cP=\cP_1\sqcup^v\cP_2$ where $\cP_j=i(F_j)$ for $j=1,2$. We have
$$
\Alg_\cP(\cC)=\Alg_{\cP_1}(\cC)\times_{\cC^k}\Alg_{\cP_2}(\cC)
$$
where the functors $g_i:\Alg_{\cP_i}(\cC)\to\cC^k$ are given by the evaluation at $v$. 
According to Lemma~\ref{lem:correctfp} the functor $g_2:\Alg_{\cP_2}(\cC)\to\cC$ induces a fibration
of the corresponding simplicial categories of fibrant-cofibrant objects. So, applying the functor of homotopy coherent nerve, we get a decomposition
$$
L(\Alg_\cP(\cC))=L(\Alg_{\cP_1}(\cC))\times_{L(\cC)^k}
L(\Alg_{\cP_2}(\cC)),
$$ 
where one of the structure maps is a categorical fibration of quasicategories, so it represents the fiber product in $\Cat$.
Since, by definition, the same decomposition holds for $\Alg_\cP(L\cC)$,
we deduce that (\ref{eq:equivalence}) is an equivalence by induction on the number of corollas in $\cP$.
 
\end{proof}

\subsubsection{Proof of Theorem~\ref{thm:SMequivalence}}
The diagram
\begin{equation}
\xymatrix{
&\Phi\ar^Y[d] \ar^i[r] & \LOp\\
&P(\Phi)\ar^L[r]&\DOp\ar_\lambda[u] 
}
\end{equation}
with $Y$ the Yoneda embedding, is commutative by~\ref{prp:psiisi}.
The composition $i_!=\lambda\circ L$ preserves colimits. By
\ref{prp:i-coop} the functor
$i:\Phi\to\LOp$ has a canonical extension to a map of anti-operads.
We will show that $i_!$ canonically extends  to a map of colax 
symmetric monoidal categories with colimits. If 
$u:\boldsymbol\Phi\to^\otimes\!\!\hat{\boldsymbol\Phi}$ is the symmetric monoidal envelope of $\boldsymbol\Phi$, 
the map of anti-operads $i:\boldsymbol\Phi\to^\otimes\!\LOp$ canonically extends to
a symmetric monoidal functor $\hat{\boldsymbol\Phi}\to\LOp$ that gives, 
by Lemma~\ref{lem:Day-univ}, a colimit preserving symmetric monoidal functor
$\Upsilon:P(\hat{\boldsymbol\Phi})\to\LOp$. The composition 
of $\Upsilon$ with $u_!:P(\Phi)\to P(\hat{\boldsymbol\Phi})$
yields a map of colax SM categories
 extending $i_!$,
see~\ref{sss:presheaves}.
Since $\Upsilon\circ u_!$ $=i_!=\lambda\circ L$ carries operadic equivalences to equivalences, 
it factors through a symmetric monoidal functor from $\DOp$ to $\LOp$ extending $\lambda$. This proves the theorem.

\subsection{} We  present below, for the convenience of the reader,
a diagram presenting some important categories and functors
appearing in the paper.

\begin{equation}
\label{eq:thediagram}
\xymatrix{
&\F \ar_{\mathrm{Yoneda}}[d]\ar^\omega[r] &\Phi\ar^{o\quad}[r]\ar_j[dl]\ar^i[dr]
&\Op(\Set)\ar^\ell[d]\\
&P(\F)& & \LOp\ar@{_{(}->}[ll]
}
\end{equation}

\appendix
\section{Shuffles of trees}

The category $\dSet$ of (set-valued) presheaves on $\Omega$ carries
the ``operadic'' model structure already mentioned in~\ref{sss:dop} above and having $\DOp$ as the underlying $\infty$-category. The tensor product on $\dSet$ does not make it a monoidal model category, however, because,
for instance, the functors $S\otimes\_$ where $S$ is a fixed tree,
do not preserve cofibrations (see~\cite{HM}, Section 4.3, for a discussion of this point). 
The smaller category $\odSet$ of presheaves on {\sl open}
trees does have a homotopically well-behaved tensor product, see~\cite{HHM}, Section 6.3.
In this appendix we explain how some of these good homotopical properties
of the open trees extend to arbitrary trees. This will imply that 
$\DOp$ is a symmetric monoidal $\infty$-category by the argument presented in Section~\ref{sec:monoidal}.

\subsection{Terminology} Recall from \ref{sss:dop} the category $\Omega$
of  trees.   For a tree $S$ we 
denote the set of its maximal edges, i.e., its leaves and stumps, by
$$\max(S).$$
A tree is open if it has no stumps. For an arbitrary tree $S$ we write $S^\circ\to S$ for its  ``interior'', obtained by chopping off the stumps. So, $S^\circ\to S$ is bijective on edges. (Warning: the assignment of $S^\circ$ to $S$ is not functorial.) If $e$ is a leaf of $S$, we denote by $S[\bar e]$ the tree obtained by adding a stump on top of $e$. We will also use the similar notation $S[\bar E]$ for a set $E$ of leaves of $S$. For example,
$$
	
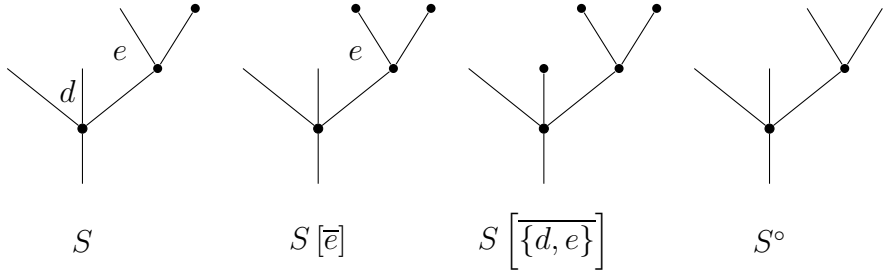
\begin{tikzpicture} 
	[level distance=8mm, 
	every node/.style={fill, circle, minimum size=.12cm, inner sep=0pt}, 
	level 1/.style={sibling distance=10mm}, 
	level 2/.style={sibling distance=10mm}, 
	level 3/.style={sibling distance=10mm}]
	
	%R_1
	\node(anchorR1)[style={color=white}] {} [grow'=up] 
	child {node(vertexR1)[draw] {} 
		child
		child
		child{node {}
			child
			child{node {}}
			}
	};
	
	%R_2
	\node[style={color=white}, right=3cm of anchorR1] {} [grow'=up] 
	child {node(vertexR2)[draw] {} 
		child
		child
		child{node {}
			child{node {}}
			child{node {}}
			}
	};
	
	%R_3
	\node[style={color=white}, right=6cm of anchorR1] {} [grow'=up] 
	child {node(vertexR3)[draw] {} 
		child
		child{node {}}
		child{node {}
			child{node {}}
			child{node {}}
			}
	};
	
	%R_4
	\node[style={color=white}, right=9cm of anchorR1] {} [grow'=up] 
	child {node(vertexR4)[draw] {} 
		child
		child
		child{node {}
			child
			child
			}
	};
	
	\tikzstyle{every node}=[]
	
	%labels
	%R1
	\node at ($(vertexR1) + (-.2cm,.5cm)$) {$d$};
	\node at ($(vertexR1) + (.5cm, 1cm)$) {$e$};
	\node at ($(vertexR1) + (0,-1.5cm)$) {$S$};
	
	%R2
	\node at ($(vertexR2) + (.5cm, 1cm)$) {$e$};
	\node at ($(vertexR2) + (0,-1.5cm)$) {$S\left[\overline{e}\right]$};
	
	%R3
	\node at ($(vertexR3) + (0,-1.5cm)$) {$S\left[\overline{\{d,e\}}\right]$};
	
	%R4
	\node at ($(vertexR4) + (0,-1.5cm)$) {$S^\circ$};
	
	\end{tikzpicture}
$$
The edge $e$ corresponds to a map $e:\eta\to S$ from the unit tree $\eta$, and $e$ extends to a map $\bar e:\bar\eta\to S[\bar e]$ where
$\bar\eta=C_0$ is the null-corolla. The tree $S[\bar e]$ is the grafting
$S\circ_e{\bar e}$ of $C_0$ onto $S$ at $e$, and the map
$$
S\sqcup^e\bar e\to S[\bar e]
$$
is a weak equivalence (a Segal map) in the model structure mentioned above. For a set $E$ of leaves, we have a similar weak equivalence
\begin{equation}
\label{eq:graftebar}
S\sqcup^E\bar E\to S[\bar E]
\end{equation}
where $S\sqcup^E\bar E$ denotes the pushout of $\sqcup_{e\in E}\eta\to S$ along $\sqcup_{e\in E}\eta\to\sqcup_{e\in E}\bar\eta$.

\subsection{Shuffles}

The $n$-fold tensor product $S_1\otimes\ldots\otimes S_n$ of a sequence of trees $S_1,\ldots, S_n$ is a union of ``shuffles'' 
(see~\cite{HM}, Section 4.4).
A shuffle $A\to S_1\otimes\ldots\otimes S_n$ is a tree whose edges are (labeled by) $n$-tuples of edges $(e_1,\ldots,e_n)$ of $S_1,\ldots,S_n$,
respectively. Not all such tuples will occur in a particular shuffle. But for us, it is important to note that the root of a shuffle $A$
is the $n$-tuple $(r_1,\ldots,r_n)$ of roots, and the set of maximal edges is {\sl exactly} the set of $n$-tuples of maximal edges in $S_i$;
that is, 
\begin{equation}
\label{eq:max}
\max(A)=\prod_{i=1}^n\max(S_i).
\end{equation}

The fact that $S_1\otimes\ldots\otimes S_n=\bigcup_{j\in J}A_j$ is the union of its
shuffles can be expressed as a finite colimit
$$
S_1\otimes\ldots\otimes S_n=\colim A_\alpha,
$$
where $\alpha$ ranges over  non-empty subsets of $J$, and
$$
A_\alpha=\bigcap _{j\in\alpha}A_j
$$
is the corresponding intersection of shuffles. Each such finite
intersection $A_\alpha$ has property (\ref{eq:max}) and each map $A_\alpha\to A_\beta$ for $\beta\subseteq\alpha\subseteq J$ is an inner face map, in fact, a map obtained by contracting a set of edges
other than the maximal edges or the root.  

The structure of the set of shuffles does not depend on the stumps that
might occur in the trees $S_i$. More precisely, for a leaf $e$ in $S_i$, there is a bijective correspondence between the shuffles of 
$S_1\otimes\ldots\otimes S_n$ and of $S_1\otimes\ldots\otimes S_i[\bar e]\otimes\ldots\otimes S_n$, given by $A\mapsto A[\bar E_i]$,
where $E_i=\{(d_1,\ldots,d_n)|d_i=e, d_j\textrm{ are all leaves of }S_j\}$. 
The same applies to the intersections of shuffles
$A_\alpha\mapsto A_\alpha[\bar E_i]$.

In particular, for trees $S_1,\ldots,S_n$ the tensor product
$S_1\otimes\ldots\otimes S_n$ can be reconstructed from the tensor product of their interiors $S^\circ_1,\ldots,S^\circ_n$: precisely, if
$S^\circ_1\otimes\ldots\otimes S^\circ_n=\bigcup_jA_j=\colim A_\alpha$,
then
 $S_1\otimes\ldots\otimes S_n=\bigcup_jA_j[\bar E]=\colim A_\alpha[\bar E]$,
 where 
$$
E=\{(e_1,\ldots,e_n)|\textrm{ each }e_j\textrm{ is a leaf of }
S_j^\circ,\textrm{ at least one}\ e_i\textrm{ is a stump of }S_i\}.
$$
From this observation one easily deduces the following lemma.
\begin{lem}
For trees $S_1,\ldots,S_n$ and the set $E$ as above, the map
$$
(S_1^\circ\otimes\ldots\otimes S_n^\circ)\sqcup^E\bar E\to
S_1\otimes\ldots\otimes S_n
$$
is a weak equivalence.
\end{lem}
\begin{proof}
Write $S_1\otimes\ldots\otimes S_n$ as the colimit $\colim A_\alpha$
of the diagram of finite intersections of shuffles. This is a Reedy
cofibrant diagram, so $S_1^\circ\otimes\ldots\otimes S_n^\circ$ is
also the homotopy colimit. Then 
$
(S_1^\circ\otimes\ldots\otimes S_n^\circ)\sqcup^E\bar E
$
is the colimit of the diagram of the $A_\alpha$ together with the inclusions
$$
\bar\eta\leftarrow\eta\stackrel{e}{\to}A_\alpha
$$
for each $e$ in $E$ and each $\alpha$. This colimit is the same as the colimit of the pushouts $A_\alpha\sqcup^E\bar E$. But $A_\alpha\sqcup^E\bar E\to A_\alpha[\bar E]$ is a weak equivalence for each $\alpha$, see~(\ref{eq:graftebar}), and $S_1\otimes\ldots\otimes S_n=\colim_\alpha
A_\alpha[\bar E]$ is again the homotopy colimit of the corresponding Reedy cofibrant diagram. Therefore, the weak equivalences  
$A_\alpha\sqcup^E\bar E\to A_\alpha[\bar E]$ for different $\alpha$ yield the one in the lemma.
\end{proof}
\subsection{Segal condition}
Consider trees $S_1,\ldots, S_n$ and a further tree $T$. Let $d$ be an 
inner edge in $T$. Cutting the tree $T$ at $d$ results in two trees 
$T^d$ and $T_d$ of which $T=T^d\circ_dT_d$ is a grafting.

\begin{prp}
\label{prp:app-segal}
The map 
$$
S_1\otimes\ldots\otimes S_n\otimes(T^d\cup T_d)\to
S_1\otimes\ldots\otimes S_n\otimes T
$$
is a weak equivalence in the operadic model structure on $\dSet$.
\end{prp}
\begin{proof}
The claim is known to hold if all the trees are open, see~\cite{HHM}, Lemma 6.3.5.
 The general case immediately follows from the lemma above, at least if $d$ itself is not a stump in $T$. If it is, 
$T_d=\bar\eta$  and $T=T^d[\bar\eta]$, and the map
in the proposition is
$$
S_1\otimes\ldots\otimes S_n\otimes(T^d\sqcup^d\bar d)\to
S_1\otimes\ldots\otimes S_n\otimes T.
$$
But in this case 
$$
S_1\otimes\ldots\otimes S_n\otimes(T^d\sqcup^d\bar d)\stackrel{\sim}{=}
(S_1\otimes\ldots\otimes S_n\otimes T^d)\sqcup^D\bar D
$$
where $D=\{(e_1,\ldots,e_n,d)|e_i \textrm{ is a leaf in }S_i\}$,
and the proposition becomes a special case of the lemma again.
\end{proof}
\subsection{Associativity}

The tensor product of dendroidal sets is not associative. For example, cosider trees $R,S$ and $T$. If $S\otimes T$ is a union of shuffles, say, $S\otimes T=\bigcup A_j$, then $R\otimes(S\otimes T)=\bigcup R\otimes A_j$ is a union of only a subset of shuffles making up $R\otimes S\otimes T$, see \cite{HM}, Section 4.4. More generally, if $S_1\otimes
\ldots\otimes S_n=\bigcup_{j\in J}A_j$ as above, then the map
\begin{equation}
\label{eq:asso-ij}
S_1\otimes\ldots\otimes(S_i\otimes\ldots\otimes S_j)\otimes\ldots\otimes S_n\longrightarrow S_1\otimes\ldots\otimes S_n
\end{equation}
is an inclusion of the form $\bigcup_{k\in K}A_k\to\bigcup_{j\in J}A_j$ where $K\subseteq J$.
\begin{prp}
\label{prp:app-asso}
Let $S_1,\ldots,S_n$ be trees. Then for any $1\leq i<j\leq n$ the map
(\ref{eq:asso-ij}) is a weak equivalence, and similarly for more
nested bracketings.
\end{prp}
\begin{proof}
This map is a weak equivalence (in fact, inner anodyne) if all the $S_i$
are open, and the same holds for more nested bracketings, see~\cite{HHM}, Lemma 6.3.6. So for general trees $S_1,\ldots, S_n$ the map
$$
S^\circ_1\otimes\ldots\otimes(S^\circ_i\otimes\ldots\otimes S^\circ_j)\otimes\ldots\otimes S^\circ_n\longrightarrow S^\circ_1\otimes\ldots\otimes S^\circ_n
$$
is a weak equivalence. As before, we can write this map as
$\colim A_\beta\to\colim A_\alpha$ where $\beta\subseteq K$ and $\alpha\subseteq J$ are nonempty  subsets, and the colimits are the colimits of Reedy cofibrant diagrams. Taking the pushout along 
$\sqcup_E\eta\to\sqcup_E\bar\eta$ for $E$ as before yields the top map in the diagram
\begin{equation}
\xymatrix{
&\colim A_\beta\sqcup^E\bar E \ar^\sim[r]\ar[d]
&\colim A_\alpha\sqcup^E\bar E \ar[d]\\
&\colim A_\beta[\bar E] \ar[r]&\colim A_\alpha[\bar E]  
}
\end{equation}
The vertical maps are colimits of grafting weak equivalences of the form (\ref{eq:graftebar}),
so the bottom map is a weak equivalence as well. But this is precisely the map in the proposition.
\end{proof}

\end{document}